\documentclass{amsart}

\usepackage[latin1]{inputenc}
\usepackage{scrtime}
\usepackage[T1]{fontenc}
\usepackage{amsmath,amssymb,amscd,latexsym,color}
\usepackage{verbatim}
\usepackage{graphicx}

\usepackage[french]{babel}

\newcommand{\N}{\mathbb N}
\newcommand{\Z}{\mathbb{Z}}
\newcommand{\Zd}{\mathbb{Z}^d}
\newcommand{\Q}{\mathbb{Q}}
\newcommand{\R}{\mathbb{R}}
\newcommand{\Rd}{\mathbb{R}^d}
\renewcommand{\P}{\mathbb{P}}
\newcommand{\E}{\mathbb{E}}

\newcommand{\Ed}{\mathbb{E}^d}

\newcommand{\Card}[1]{\vert #1 \vert}
\newcommand{\1}{1\hspace{-1.3mm}1}

\newcommand{\bor}[1][]{\mathcal{B}(#1)}

\newcommand{\Ebarre}{\overline{\mathbb{E}}}
\newcommand{\Pbarre}{\overline{\mathbb{P}}}
\newcommand{\Qbarre}{\overline{\mathbb{Q}}}
\newcommand{\D}{\mathcal{D}}
\newcommand{\T}[2]{{#1}.{#2}}

\renewcommand{\epsilon}{\varepsilon}
\renewcommand{\phi}{\varphi}
\renewcommand{\limsup}{\overline{\lim}}
\renewcommand{\liminf}{\underline{\lim}}

\newcommand{\miniop}[3]{
\renewcommand{\arraystretch}{0.6}
\begin{array}{c}
{\scriptstyle #1}\\
#2\\
{\scriptstyle #3}
\end{array}
\renewcommand{\arraystretch}{1}}

\begin{document}

{
\newtheorem{theorem}{Théorème}[section]
\newtheorem{conjecture}[theorem]{Conjecture}

}
\newtheorem{lemme}[theorem]{Lemme}
\newtheorem{defi}[theorem]{Définition}
\newtheorem{coro}[theorem]{Corollaire}
\newtheorem{rem}[theorem]{Remarque}
\newtheorem{prop}[theorem]{Proposition}
\newcommand{\finpreuve}{}
\newcommand{\debutpreuve}{Démonstration }

\title[Processus de contact en environnement aléatoire]{La forme asymptotique du processus de contact en environnement aléatoire}

{
\author{Olivier Garet}
\address{Institut \'Elie Cartan Nancy (mathématiques)\\
Université Henri Poincaré Nancy 1\\
Campus Scientifique, BP 239 \\
54506 Vandoeuvre-lès-Nancy  Cedex France\\}
\email{Olivier.Garet@iecn.u-nancy.fr}
\author{Régine Marchand}
\email{Regine.Marchand@iecn.u-nancy.fr}

}

\def\motsclefs{Random growth, contact process, random environment, shape theorem}


\begin{abstract}
Le but de cet article est d'établir des théorèmes de forme asymptotique pour le processus de contact en environnement aléatoire stationnaire, généralisant ainsi des résultats connus pour le processus de contact en environnement déterministe. En particulier, on montre  que pour presque toute réalisation de l'environnement aléatoire et pour presque toute réalisation du processus de contact telle que le processus survit, l'ensemble  $H_t$  des points qui ont été occupés avant le temps $t$ est tel que $H_t/t$ converge vers un compact qui ne dépend que de la loi de l'environnement.
La preuve utilise un nouveau théorème ergodique presque sous-additif.

\vspace{0.1cm} 
\emph{Mots clefs}: croissance aléatoire, processus de contact, environnement aléatoire, théorème ergodique sous-additif, théorème de forme asymptotique
\vspace{0.2cm}

\noindent\textsc{Abstract.} \foreignlanguage{english}{
The aim of this article is to prove asymptotic shape theorems for the contact process in stationary random environment. These theorems generalize known results for the classical contact process. In particular, if $H_t$ denotes the set of already occupied sites at time $t$, we show that for almost every environment, when the contact process survives, the set  $H_t/t$ almost surely converges to a compact set that only depends on the law of the environment. To this aim, we prove a new almost subadditive ergodic theorem.
}

\vspace{0.1cm}
\foreignlanguage{english}{\emph{Key words}: Random growth, contact process, random environment, subadditive ergodic theorem, shape theorem}

\end{abstract}


{\maketitle
}
\setcounter{tocdepth}{1}

\section{Introduction}
Le but de cet article est d'obtenir un théorème de forme asymptotique pour le processus de contact dans un environnement aléatoire en dimension~$d$. 
Dans notre cas, l'environnement est donné par une famille de variables
aléatoires 
$(\lambda_e)_{e\in\Ed}$ indexée par l'ensemble $\Ed$ des arêtes du réseau cubique $\Zd$, la variable aléatoire $\lambda_e$ représentant 
le taux de naissance sur l'arête~$e$ tandis que les taux de mort sont tous
égaux à 1. La loi des  $(\lambda_e)_{e\in\Ed}$ est supposée stationnaire et ergodique.

Notre principal résultat  est le suivant: si on suppose que les $(\lambda_e)_{e\in\Ed}$ prennent leurs valeurs au dessus
de $\lambda_c(\Zd)$, le paramètre critique pour la possibilité de survie du processus de contact ordinaire sur $\Zd$, alors il existe une norme $\mu$ sur $\Rd$ telle que, pour presque tout environnement $\lambda=(\lambda_e)_{e\in\Ed}$, l'ensemble $H_t$ des points déjà infectés au moins une fois au temps $t$ vérifie
$$\Pbarre_{\lambda}\left(\exists T>0 \; t\ge T\Longrightarrow (1-\epsilon)tA_{\mu}\subset H_t \subset (1+\epsilon)tA_{\mu}\right)=1,$$
où $A_{\mu}$ est la boule unité de la norme $\mu$, et 
$\Pbarre_{\lambda}$ la loi du processus de contact en environnement $\lambda$,
conditionné à survivre. On retrouve donc un théorème de forme asymptotique analogue à celui gouvernant la croissance du processus de contact surcritique sur $\Zd$ en environnement déterministe.

Jusqu'à présent, les travaux concernant le processus de contact en environnement aléatoire  sont essentiellement consacrés à la détermination de conditions assurant la survie (Liggett~\cite{MR1159569}, Andjel~\cite{MR1203176}, Newman et Volchan~\cite{MR1387642}), ou l'extinction (Klein~\cite{MR1303643}) du processus de contact. 
Par ailleurs, la plupart d'entre eux traitent de la dimension un. 
Ainsi, Bramson, Durrett et Schonman~\cite{MR1112403} montrent qu'en dimension un et en environnement aléatoire, une croissance sous-linéaire est possible. 
Ils conjecturent en revanche qu'un théorème de forme asymptotique devrait pouvoir être obtenu en dimension $d\ge 2$, dès que la survie du processus de contact est possible.

Rappelons les deux étapes de la preuve du résultat de forme asymptotique dans le cas du processus de contact en environnement déterministe:
\begin{itemize}
 \item En 1982, Durrett et Griffeath~\cite{MR656515} montrent le résultat pour les grandes valeurs du taux de naissance $\lambda$. Pour ces grandes valeurs, ils obtiennent des estimées garantissant essentiellement que la croissance est d'ordre linéaire, puis utilisent des techniques (presque) sous-additives pour en déduire le résultat de forme asymptotique
\item Plus tard, Bezuidenhout et Grimmett~\cite{MR1071804} montrent qu'un processus de contact surcritique sur $\Zd$, vu à grande échelle, domine stochastiquement une percolation orientée surcritique. Ils indiquent également comment leur construction peut être utilisée pour prolonger le résultat de forme asymptotique dans toute la zone surcritique. Cette dernière étape, qui consiste à montrer que les estimées utilisées dans~\cite{MR656515} s'étendent à tout le régime surcritique est faite en détail par Durrett dans~\cite{MR940469}. 
\end{itemize}

Dans le cas de l'environnement aléatoire, on retrouve ces deux aspects du problème. Le défi certainement le plus difficile consiste à montrer que dès que la survie est possible, la croissance du processus de contact est d'ordre linéaire; ceci correspondrait à montrer l'équivalent en environnement aléatoire du résultat de Bezuidenhout et Grimmett~\cite{MR1071804}. Un autre problème est de montrer que sous des conditions garantissant que la croissance est d'ordre linéaire, on peut obtenir un théorème de forme asymptotique: c'est l'analogue de Durrett et Griffeath~\cite{MR656515}, et le problème auquel nous nous intéressons ici. 

Nous avons ainsi choisi d'imposer des conditions sur l'environnement aléatoire permettant d'obtenir, à l'aide de techniques classiques, des estimées similaires à celles requises dans~\cite{MR656515}: en effet, la démonstration de l'existence d'une forme asymptotique en milieu aléatoire à partir de ces estimées recèle déjà en elle-même de sérieuses difficultés.

En général, les théorèmes de forme asymptotique pour des modèles
de croissance se prouvent grâce à la théorie des processus sous-additifs initiée par Hammersley et Welsh~\cite{MR0198576}, et en particulier grâce
au théorème ergodique sous-additif de Kingman~\cite{MR0356192} et à ses différentes extensions. L'exemple le plus évident est certainement celui 
du théorème de forme asymptotique pour la percolation de premier passage sur $\Zd$ (voir aussi les diverses extensions de ce modèle: Boivin~\cite{MR1074741}, Garet et Marchand~\cite{MR2085613}, Vahidi-Asl et Wierman~\cite{MR1166620}, Howard et Newmann~\cite{MR1452554}, Howard~\cite{MR2023652}, Deijfen~\cite{MR1970474}).

Parmi les modèles relevant
de la théorie sous-additive, on peut distinguer deux familles. La première, et la plus
fréquemment étudiée, est celle des modèles permanents: la forme occupée au 
temps $t$ ne fait que croître et il n'y a pas d'extinction possible (modèles de Richardson~\cite{MR0329079}, modèle des grenouilles de Alves et al.~\cite{MR1910638,MR1893139} et des marches aléatoires branchantes de Comets et Popov~\cite{MR2303944}). Dans ces modèles, la partie essentielle du travail consiste à montrer que la croissance est au moins linéaire, la sous-additivité permettant alors
d'obtenir la convergence désirée.

La seconde famille est celle des modèles non permanents, autrement dit ceux où l'extinction est possible. Dans ce cas, c'est en
conditionnant par la survie que l'on espère obtenir un théorème
de forme asymptotique.
Les difficultés induites par la possibilité d'extinction ont
été soulignées dès la genèse de la théorie sous-additive, en particulier
par Hammersley lui-même~\cite{MR0370721}: en effet, si l'on veut démontrer que les
temps d'atteinte $(t(x))_{x\in\Zd}$ des différents points
du réseau sont tels que $t(nx)/n$ converge, la théorie de Kingman requiert 
que la famille des  variables aléatoires $t(x)$ est stationnaire 
(en un sens à préciser) et intégrable. Bien sûr, si l'extinction est possible,
il n'y a pas intégrabilité puisque les temps d'atteinte peuvent être infinis.
En revanche, si l'on conditionne par la survie, les propriétés d'indépendance, de stationnarité, voire de sous-additivité peuvent être perdues. Un premier lemme de presque sous-additivité est
proposé par Kesten dans la discussion de l'article de Kingman~\cite{MR0356192},
puis étendu par Hammersley~\cite{MR0370721} (page 674). Plus tard, on trouvera d'autres
types d'hypothèses (voir par exemple Derriennic~\cite{MR704553}, Derriennic et Hachem~\cite{MR939537}, et Schürger~\cite{MR833959,MR1127716}). 

Le processus de contact fait clairement partie de la seconde famille.
C'est sur le lemme de Kesten-Hammersley que s'appuient Bramson et Griffeath~\cite{MR606980,MR578279}, puis Durrett et Griffeath~\cite{MR656515} pour démontrer
leurs théorèmes de forme asymptotique. Cependant, leur preuve, quoique partiellement corrigée dans~\cite{MR940469}, contient un certain nombre d'erreurs liées au conditionnement. La stratégie que nous développons, différente du fait de l'environnement aléatoire, offre 
une preuve alternative du théorème de forme asymptotique pour le processus de contact en environnement déterministe. 

Bien-sûr, le caractère aléatoire de l'environnement
induit des difficultés supplémentaires. Pour parler simplement, si l'on
travaille à environnement fixé, toute stationnarité spatiale est perdue.
En revanche, si l'on travaille en environnement moyenné, c'est le caractère
markovien du processus de contact qui fait défaut.
Le lemme de Kesten et Hammersley ne peut donc pas être utilisé directement
puisque ce dernier réclame à la fois de la stationnarité et une forme
d'indépendance.
Nous introduisons ici une nouvelle quantité $\sigma(x)$, que l'on peut voir comme un temps de régénération, et qui représente un  moment où le site $x$ est occupé par un individu
dont la descendance est infinie. Ce $\sigma$ vérifie certaines propriétés de stationnarité et de presque sous-additivité qui faisaient défaut au temps d'atteinte $t(x)$ et qui permettront de reformuler le problème dans un cadre de théorie ergodique presque sous-additive. Nous établissons alors, avec des techniques inspirées de Liggett, un théorème ergodique presque sous-additif général, qui nous permet d'obtenir le théorème de forme asymptotique pour $\sigma$.
Finalement, en contrôlant l'écart  entre le temps d'atteinte $t(x)$ et $\sigma(x)$, on pourra transférer à $t$ les résultats obtenus pour $\sigma$.

\section{Modèle et résultats}
\label{construction}

\subsection{Environnement}

Dans tout l'article, on notera $\|.\|_1$ et $\|.\|_{\infty}$ les normes sur $\Rd$ respectivement définies par $\|x\|_1=\sum_{i=1}^d |x_i|$ et $\|x\|_{\infty}=\miniop{}{\max}{1\le i\le d} |x_i|$. La notation $\|.\|$ sera utilisée pour désigner une norme quelconque.

Soit $\lambda_{\min}$ et $\lambda_{\max}$ deux réels fixés,
avec $\lambda_c(\Zd)<\lambda_{\min}\le\lambda_{\max}$, où 
$\lambda_c(\Zd)$ est le taux de naissance critique pour la survie du processus de contact usuel sur $\Zd$.
Dans toute la suite, on se limitera à l'étude du processus de contact en environnement aléatoire avec des taux de naissance $\lambda=(\lambda_e)_{e \in \Ed}$ appartenant à
l'ensemble des environnements $\Lambda=[\lambda_{\min},\lambda_{\max}]^{\Ed}$. Un environnement est donc une collection $\lambda=(\lambda_e)_{e \in \Ed} \in \Lambda$.

\medskip
Soit $\lambda \in \Lambda$ un environnement fixé. Le processus de contact $(\xi_t)_{t\ge 0}$  dans l'environnement $\lambda$ est un processus de Markov homogène qui prend ses valeurs dans l'ensemble $\mathcal{P}(\Zd)$ des parties de $\Zd$, que l'on  identifiera parfois à l'ensemble $\{0,1\}^{\Zd}$. Ainsi, on s'autorisera les deux écritures
$$z \in \xi_t \text{ ou } \xi_t(z)=1.$$
Si $\xi_t(z)=1$, on dit que le site $z$ est occupé, tandis que si $\xi_t(z)=0$, on dit que le site $z$ est vide.
Le processus évolue de la façon suivante:
\begin{itemize}
\item un site occupé devient vide à taux $1$,
\item un site $z$ vide devient infecté au taux
$\displaystyle \sum_{\|z-z'\|_1=1} \xi_t(z')\lambda_{\{z,z'\}},$
\end{itemize}
ces différentes évolutions étant indépendantes les unes des autres. 
Dans la suite, on notera $\D$ l'ensemble des fonctions càdlàg de $\R_{+}$ dans $\mathcal{P}(\Zd)$: c'est l'espace canonique pour les processus de Markov admettant $\mathcal{P}(\Zd)$ comme espace d'état.

Pour définir le processus de contact en environnement $\lambda\in\Lambda$, on utilise la construction de Harris~\cite{MR0488377} des processus de Markov additifs à valeurs dans les parties de $\Zd$. Elle permet de coupler des processus de contact partant de configurations différentes, en les construisant à partir d'une même collection de mesures de Poisson sur $\R_+$.

\subsection{Construction de la famille de mesures de Poisson}
Sur $\R_+$ muni de sa tribu borélienne $\mathcal B(\R_+)$, on considère l'ensemble $M$ constitué des mesures ponctuelles $m=\sum_{i=0}^{+\infty} \delta_{t_i}$ localement finies dont tous les atomes sont de masse $1$. On munit cet ensemble de la tribu $\mathcal M$ engendrée par les applications $m\mapsto m(B)$, où $B$ décrit l'ensemble des boréliens de $\R_+$.

On définit alors l'espace mesurable $(\Omega, \mathcal F)$ par
$$\Omega=M^{\Ed}\times M^{\Zd} \text{ et } \mathcal F=\mathcal{M}^{\otimes \Ed} \otimes \mathcal{M}^{\otimes \Zd}.$$
Sur cet espace, on considère la famille de probabilités $(\P_{\lambda})_{\lambda\in\Lambda}$  définies comme suit: pour tout $\lambda=(\lambda_e)_{e \in \Ed} \in \Lambda$, 
$$\P_{\lambda}=\left(\bigotimes_{e \in \Ed} \mathcal{P}_{\lambda_{e}}\right) \otimes \mathcal{P}_1^{\otimes\Zd},$$
où, pour chaque $\lambda\in\R_+$, $\mathcal{P}_{\lambda}$ est la loi d'un processus ponctuel de Poisson sur $\R_+$ d'intensité $\lambda$. Si $\lambda \in \R_+$, on écrit plutôt $\P_\lambda$ (au lieu de $\P_{(\lambda)_{e \in \Ed}}$) pour la loi en environnement déterministe avec taux de naissance $\lambda$ en chaque arête.

Pour tout $t\ge 0$, on note $\mathcal{F}_t$ la tribu engendrée par les applications $\omega\mapsto\omega_e(B)$ et $\omega\mapsto\omega_z(B)$, où $e$ décrit $\Ed$, $z$ décrit $\Zd$, et $B$ décrit l'ensemble des boréliens de $[0,t]$.
 
\subsection{La construction graphique du processus de contact} Cette construction est très détaillée dans l'article de Harris~\cite{MR0488377}; nous ne donnons  ici qu'une description informelle. Soit $\omega=((\omega_e)_{e \in \Ed}, (\omega_z)_{z \in \Zd}) \in \Omega$. Au dessus de chaque site $z \in \Zd$, on trace un axe temporel $\R_+$, et on marque une croix aux instants donnés par $\omega_z$. Au dessus de chaque arête $e \in \Ed$, on trace aux instants donnés par $\omega_e$ un segment horizontal entre les deux extrémités de l'arête. Un chemin ouvert suit les axes temporels au dessus des sites sans pouvoir traverser les croix, et emprunte les segments horizontaux pour passer d'un axe dessiné au dessus d'un site à l'axe dessiné au dessus d'un site voisin. Si on pense le processus de contact en termes de propagation d'une infection, un chemin ouvert est un trajet  possible de l'infection d'un site par un autre.
Pour $x,y \in \Zd$ et $t \ge 0$, on dit alors que $\xi_t^x(y)=1$ si et seulement si il existe un chemin ouvert de $(x,0)$  à $(y,t)$, puis on définit:
\begin{eqnarray}
\xi_t^x & = & \{y \in \Zd: \; \xi_t^x(y)=1\}, \nonumber \\
\text{et, pour tout $A \in \mathcal P(\Zd)$,}\quad  \xi_t^A & = & \bigcup_{x \in A} \xi_t^x. \label{additivite}
\end{eqnarray}
En particulier, on a immédiatement 
$(A \subset B)  \Rightarrow (\forall t \ge 0\quad \xi_t^A \subset \xi_t^B)$.

Quand $\lambda\in\R_{+}^*$, Harris prouve que sous $\P_{\lambda}$, le processus $(\xi^A_t)_{t \ge 0}$ est le processus de contact avec taux de naissance constant $\lambda$, partant de la configuration initiale $A$. Il n'est pas difficile d'adapter la preuve pour voir que, si $\lambda \in \Lambda$, sous $\P_\lambda$, le processus $(\xi^A_t)_{t \ge 0}$ est ce que nous avons appelé le processus de contact en environnement $\lambda$, partant de la configuration initiale $A$. 
Ce processus est fellérien, et jouit donc de la propriété de Markov forte. 

\subsection{Translations temporelles}
Pour $t \ge 0$, on définit l'opérateur de translation $\theta_t$ sur une mesure ponctuelle $m=\sum_{i=1}^{+\infty} \delta_{t_i}$ sur $\R_+$ par
$$\theta_t m=\sum_{i=1}^{+\infty} \1_{\{t_i\ge t\}}\delta_{t_i-t}.$$
La translation $\theta_t$ induit de manière naturelle une opération sur $\Omega$, que l'on note encore $\theta_t$: pour tout $\omega \in \Omega$, on pose
$$ \theta_t \omega=((\theta_t \omega_e)_{e \in \Ed}, (\theta_t \omega_z)_{z \in \Zd}).$$
Les mesures ponctuelles de Poisson  étant toutes invariantes par translation, l'opérateur~$\theta_t$ laisse toutes les probabilités $\P_{\lambda}$ invariantes. La propriété de semi-groupe du processus de contact a ici une version plus forte trajectorielle: pour tout $A \subset \Zd$, pour tous $s,t \ge 0$, pour tout $\omega \in \Omega$, on a l'identité
\begin{equation}
\label{semigroupe}
\xi_{t+s}^A(\omega) =  \xi_{s}^{\xi_t^A(\omega)}(\theta_t\omega)=\xi_{s}^{\centerdot}(\theta_t\omega)\circ \xi_t^A(\omega),
\end{equation}
qui peut s'exprimer sous la forme markovienne classique
$$\forall B\in\mathcal{B}(\mathcal{D})\quad\P((\xi_{t+s}^A)_{s \ge 0} \in B| \mathcal F_t)  =  \P((\xi_{s}^{\centerdot})_{s \ge 0} \in B) \circ \xi^A_t.$$

On a l'analogue pour la propriété de Markov forte: si $T$ est un temps d'arrêt adapté à la filtration $(\mathcal{F}_t)_{t\ge 0}$, alors, sur l'événement $\{T<+\infty\}$,
\begin{eqnarray*}
 \xi_{T+s}^A(\omega)& = & \xi_{s}^{\xi_T^A(\omega)}(\theta_T\omega),\\
\nonumber\forall B\in\mathcal{B}(\mathcal{D})\quad \P((\xi_{T+s}^A)_{s \ge 0} \in B| \mathcal F_T) & = & \P((\xi_{s}^{\centerdot})_{s \ge 0} \in B) \circ \xi^A_T.
\end{eqnarray*}
Rappelons que $\mathcal{F}_T$ désigne la tribu des événements déterminés au temps $T$ définie par 
$$\mathcal{F}_T=\{B\in\mathcal{F}:\quad \forall t\ge 0\quad B\cap \{T\le t\}\in\mathcal{F}_t\}.$$

\subsection{Translations spatiales}
On peut faire  agir $\Zd$ à la fois sur le processus et 
sur l'environnement.
L'action sur le processus consiste à changer le point de vue de l'observateur de l'évolution du processus : pour $x \in \Zd$, on définit l'opérateur de translation~$T_x$ par
$$\forall \omega \in \Omega\quad T_x \omega=(( \omega_{x+e})_{e \in \Ed}, ( \omega_{x+z})_{z \in \Zd}),$$
où l'on a convenu que $x+e$ était la translatée de vecteur $x$ de l'arête $e$.

Par ailleurs, pour tout environnement $\lambda\in\Lambda$, on considérera
l'environnement translaté $\T{x}{\lambda}$ défini par $(\T{x}{\lambda})_e=\lambda_{x+e}$.
Ces deux actions sont duales au sens suivant: pour tout $\lambda \in \Lambda$, pour tout $x \in \Zd$, on a
\begin{eqnarray}
\label{translationspatiale}
\forall A\in\mathcal{F}\quad\P_{\lambda}(T_x \omega \in A) & = & \P_{\T{x}{\lambda}}(\omega \in A); 
\end{eqnarray}
en particulier, la loi de $\xi^x$ sous $\P_\lambda$ est \'egale \`a la loi de $\xi^0$ sous $\P_{x.\lambda}$.

\subsection{Temps d'atteinte essentiels et transformations associées}
Pour $A \subset \Zd$, on définit le temps de vie $\tau^A$ du processus issu de $A$, 
$$\tau^A=\inf\{t\ge0: \; \xi_t^A=\varnothing\}. $$
Pour $A \subset \Zd$ et $x \in \Zd$,  on définit également l'instant $t^A(x)$ de première infection du point $x$ en partant de $A$: 
$$t^A(x)=\inf\{t\ge 0: \; x \in \xi_t^A\}.$$
Si $y\in\Zd$, on note  $t^y(x)$ pour  $t^{\{y\}}(x)$. De même, on écrira simplement $t(x)$ pour $t^0(x)$.

On introduit alors la quantité $\sigma(x)$ qui s'avèrera primordiale dans 
la suite: il s'agit d'un  instant où naît au site $x$, dans le processus issu de $0$,  un point
dont la descendance ne s'éteint pas. On le définit par une famille de temps d'arrêts comme suit: on pose $u_0(x)=v_0(x)=0$ et on définit  par récurrence deux  suites croissantes de temps d'arrêt $(u_n(x))_{n \ge 0}$ et $(v_n(x))_{n \ge 0}$ avec
$u_0(x)=v_0(x)\le u_1(x)\le v_1(x)\le u_2(x)\dots$  de la façon suivante:
\begin{itemize}
\item Supposons avoir construit $v_k(x)$. On pose
$
u_{k+1}(x)  =\inf\{t\ge v_k(x): \; x \in \xi^0_t \}.
$ \\
Si $v_k(x)<+\infty$, alors $u_{k+1}(x)$ représente le premier instant après $v_k(x)$ où le point $x$ est à nouveau occupé; sinon $u_{k+1}(x)=+\infty$.
\item Supposons avoir construit $u_k(x)$, avec $k \ge 1$. On pose
$v_k(x)=u_k(x)+\tau^x\circ \theta_{u_k(x)}$.\\
Si $u_k(x)<+\infty$, le temps $\tau^x\circ \theta_{u_k(x)}$ représente la durée de vie du processus de contact démarrant en $x$ à l'instant $u_k(x)$; sinon $v_k(x)=+\infty$.
\end{itemize}
On pose alors
\begin{equation}
\label{definitiondeK}
K(x)=\min\{n\ge 0: \; v_{n}(x)=+\infty \text{ ou } u_{n+1}(x)=+\infty\}.
\end{equation}
Cette quantité représente le nombre d'étapes avant que l'on arrête le procédé: on s'arrête soit parce qu'on trouve un $v_n(x)$ infini, ce qui correspond à trouver un instant $u_n(x)$ où le point $x$ est à la fois occupé et départ d'une descendance infinie, soit parce qu'on trouve un $u_{n+1}(x)$ infini, ce qui correspond au fait qu'après $v_n(x)$, le point $x$ n'est plus jamais occupé.

On pose alors $\sigma(x)=u_{K(x)}$.

Nous l'appellerons le \emph{temps d'atteinte essentiel} de $x$. Il est bien sûr plus grand que le temps d'atteinte $t(x)$. 
On verra que la quantit\'e $K(x)$ est presque sûrement finie, ce qui fait que $\sigma(x)$ est bien défini.
Conjointement, on définit la transformation $\tilde \theta_x$ de $\Omega$ dans lui-même par:
\begin{equation*}
\tilde \theta_x = 
\begin{cases} T_{x} \circ \theta_{\sigma(x)} & \text{si $\sigma(x)<+\infty$,}
\\
T_x &\text{sinon,}
\end{cases}
\end{equation*}
ou, si l'on se veut plus explicite:
\begin{equation*}
(\tilde \theta_x)(\omega) = 
\begin{cases} T_{x} (\theta_{\sigma(x)(\omega)} \omega) & \text{si $\sigma(x)(\omega)<+\infty$,}
\\
T_x (\omega) &\text{sinon.}
\end{cases}
\end{equation*}

 Nous allons travailler principalement avec le temps d'atteinte essentiel $\sigma(x)$ qui possède, contrairement à $t(x)$, de bonnes propriétés d'invariance en environnement conditionné à survivre.
Nous verrons qu'on peut aussi contrôler la différence entre $\sigma(x)$ et $t(x)$, ce qui permettra de transposer les résultats obtenus pour $\sigma(x)$ à $t(x)$.

\subsection{Processus de contact en environnement aléatoire conditionné à survivre}
Nous allons maintenant nous placer en environnement aléatoire. Pour toute la suite, on fixe une mesure de probabilité $\nu$ sur l'ensemble des environnements $\Lambda=[\lambda_{\min},\lambda_{\max}]^{\Ed}$. On suppose que $\nu$ est stationnaire et ergodique sous
l'action de $\Zd$. Bien évidemment, cela contient le cas d'un environnement déterministe classique avec un taux de naissance constant $\lambda>\lambda_c(\Zd)$: il suffit de prendre pour $\nu$ la masse de Dirac~$(\delta_{\lambda})^{\otimes \Ed}$.

Pour $\lambda \in \Lambda$, on définit la probabilité ${\Pbarre}_\lambda$ 
sur $(\Omega, \mathcal F)$ par
$$\forall E\in\mathcal{F}\quad {\Pbarre}_\lambda(E)=\P_\lambda(E|\tau^0=+\infty).$$
C'est la loi de la famille des  processus ponctuels de Poisson, conditionnés à ce que le processus de contact issu de $0$ survive.
Sur le même espace $(\Omega, \mathcal F)$, on définit alors la probabilité moyennée (annealed) $\Pbarre$ par
$$\forall E\in\mathcal{F}\quad {\Pbarre}(E)=\int_\Lambda {\Pbarre}_\lambda(E)\ d\nu(\lambda).$$
Autrement dit, l'environnement $\lambda=(\lambda_e)_{e \in \Ed}$ dans lequel le processus de contact évolue est une variable aléatoire de loi $\nu$, et 
c'est sous cette dernière probabilité $\Pbarre$ que l'on va chercher à établir le théorème de forme asymptotique.

Il aurait pu sembler plus naturel de travailler avec la probabilit\'e suivante:
$$\forall E\in\mathcal{F}\quad \hat{\P}(E)=\P(E|\tau^0=+\infty)=
\frac{\int \Pbarre_{\lambda}(E)\P_\lambda(\tau^0=+\infty) d\nu(\lambda)}{\int \P_\lambda(\tau^0=+\infty) d\nu(\lambda)}.$$
Notre choix de ne consid\'erer que des environnements uniform\'ement surcritiques assure cependant que $\Pbarre$ et $\hat{\P}$ sont \'equivalentes. Le th\'eor\`eme de forme asymptotique que nous \'enon\c{c}ons presque s\^urement sous $\Pbarre$ peut ainsi \^etre \'enonc\'e presque s\^urement sous $\hat{\P}$.

\subsection{Organisation de l'article et résultats}
Dans la section~\ref{sigma}, on 
établit les propriétés d'invariance et d'ergodicité. On montre en particulier le théorème suivant:
\begin{theorem}
\label{systemeergodique}
Pour tout $x\in\Zd\backslash\{0\}$, le système $(\Omega,\mathcal{F},\Pbarre,\tilde{\theta}_x)$ est ergodique.
\end{theorem}
Dans la section~\ref{controlediff}, on étudie les propriétés d'intégrabilité des $(\sigma(x))_{x\in\Zd}$; on contrôle l'écart entre $\sigma(x)$ et $t(x)$ ainsi que le défaut de sous-additivité de $\sigma$:
\begin{theorem}
\label{presquesousadditif}
Il existe des constantes positives $A_{\ref{epresquesousadditif}},B_{\ref{epresquesousadditif}}$ telles que 
pour tout $\lambda\in\Lambda$, pour tous $x,y\in\Zd$, 
\begin{equation}
\forall t>0 \quad \Pbarre_\lambda(\sigma(x+y)-(\sigma(x)+\sigma(y)\circ\tilde{\theta}_x)\ge t)\le A_{\theequation}\exp(-B_{\theequation}\sqrt{t}).
\label{epresquesousadditif}
\end{equation}
\end{theorem}
Ainsi, le d\'efaut de sous-additivit\'e de $\sigma$ est très faible; en particulier il ne dépend pas des points consid\'er\'es.
Alors, s'inspirant des méthodes de Kingman~\cite{MR0438477} et Liggett~\cite{MR806224}, on montre dans la section~\ref{forme} que pour tout $x$ dans $\Zd$, le rapport $\frac{\sigma(nx)}n$ converge $\Pbarre$ presque sûrement vers un réel $\mu(x)$. 
La fonctionnelle $x\mapsto \mu(x)$ se prolonge en une norme sur $\Rd$, qui va caractériser la forme asymptotique. Dans la suite, on notera $A_{\mu}$ la boule unité pour $\mu$.
On définit 
\begin{eqnarray*}
H_t & = & \{x\in\Zd: \; t(x)\le t\},\\
G_t & = & \{x\in\Zd: \; \sigma(x)\le t\},\\
K'_t & = & \{x\in\Zd: \;\forall s\ge t \quad \xi^0_s(x)=\xi^{\Zd}_s(x)\},
\end{eqnarray*}
et on désigne par $\tilde{H}_t,\tilde{G}_t,\tilde{K}'_t$ les versions grossies 
des ensembles $H_t,G_t,K'_t$: 
$$\tilde{H}_t=H_t+[0,1]^d, \; \tilde{G}_t=G_t+[0,1]^d \text{ et } \tilde{K}'_t=K'_t+[0,1]^d.$$
On peut alors démontrer les résultats suivants:

\begin{theorem}[Théorème de forme asymptotique]
\label{thFA}
Pour tout $\epsilon>0$, avec probabilité~1 sous $\Pbarre$, pour tout $t$ suffisamment grand,
\begin{equation}
\label{leqdeforme}
(1-\epsilon)A_{\mu}\subset \frac{\tilde K'_t\cap \tilde G_t}t\subset \frac{\tilde G_t}t\subset\frac{\tilde H_t}t\subset (1+\epsilon)A_{\mu}.
\end{equation}
\end{theorem}
L'ensemble $K'_t\cap G_t$ est la zone couplée du processus. Notons que comme la littérature existante n'a pas fait jouer de rôle particulier à $\sigma(x)$, les théorèmes de forme asymptotique considèrent plutôt la quantité $K_t\cap H_t$,
avec $$K_t=\{x\in\Zd: \;  \xi^0_t(x)=\xi^{\Zd}_t(x)\}.$$ Notre résultat implique également le théorème de forme asymptotique pour $K_t\cap H_t$, car
$K'_t\cap G_t\subset K_t\cap H_t\subset H_t$.

Remarquons que le théorème de forme asymptotique peut se reformuler sous la forme "quenched" suivante: pour $\nu$ presque tout environnement, on sait 
que sur l'événement ``le processus de contact survit'', sa
croissance est presque sûrement gouvernée par~(\ref{leqdeforme}) pour tout $t$ suffisamment grand. 
Dans le même ordre d'idées, on peut retrouver pour $\nu$ presque tout environnement le résultat de convergence en loi suivant:
\begin{theorem}[Théorème de convergence complète]
\label{thCC}
Pour tout $\lambda\in\Lambda$, le processus de contact dans l'environnement $\Lambda$ admet une mesure invariante maximale $m_\lambda$ qui est caractérisée par
$$\forall A\subset \Zd, \Card{A}<+\infty\quad m_\lambda(\omega\supset A)=\lim_{t\to +\infty}\P_{\lambda}(\xi^{\Zd}_t\supset A).$$
Alors, pour toute partie  finie $A$ de $\Zd$ et pour $\nu$ presque tout $\lambda$, on a
$$\P^A_{\lambda,t} \Longrightarrow \P_{\lambda}(\tau^A<+\infty)\delta_{\varnothing}+\P_{\lambda}(\tau^A=\infty)m_{\lambda},$$
où $\P^A_{\lambda,t}$ est la loi de $\xi^A_t$ sous $\P_{\lambda}$ et $\Longrightarrow$ représente la convergence en loi.
\end{theorem}
La preuve de ce théorème ne demandant pas d'idées nouvelles, on se contentera d'en donner l'ingrédient principal à la fin de la section~\ref{restartestimees}.

Pour démontrer le théorème de forme asymptotique, on aura besoin de quelques
contrôles exponentiels. Dans toute la suite, on note
$$B_r^x=\{y \in  \Zd: \; \|y-x\|_{\infty} \le r\},$$
et on note plus simplement $B_r$ au lieu de $B_r^0$.

\begin{prop}
\label{propuniforme}
Il existe des constantes strictement positives $A,B,M,c,\rho$ telles que pour tout
$\lambda\in\Lambda$, pour tout  $y \in \Zd$, pour tout $ t\ge0$
\begin{eqnarray}
\P_\lambda(\tau^0=+\infty) & \ge & \rho,
\label{uniftau} \\
\P_\lambda(H^0_t \not\subset B_{Mt} ) & \le & A\exp(-Bt), 
\label{richard} \\
\P_\lambda ( t<\tau^0<+\infty) &\le&  A\exp(-Bt), \label{grosamasfinis} \\
 \P_{\lambda}\left( t^0(y)\ge \frac{\|y\|}c+t,\; \tau^0=+\infty \right) & \le & A\exp(-Bt),
\label{retouche}\\
\P_{\lambda}(0\not\in K'_t, \; \tau^0=+\infty) &\le &A\exp(-B t).
\label{petitsouscouple} 
\end{eqnarray}
\end{prop}

On dispose déjà des estimées de la proposition~\ref{propuniforme} en environnement déterministe homogène $\lambda$, pour $\lambda>\lambda_c(\Zd)$. Le résultat pour les grands $\lambda$ est dû à Durrett et Griffeath~\cite{MR656515}. L'extension à tout le régime surcritique est rendue possible grâce au travail de Bezuidenhout et Grimmett~\cite{MR1071804}. Pour les détails de la preuve de l'inégalité~(\ref{grosamasfinis}), qui en est le point essentiel, voir par exemple l'article de revue de Durrett~\cite{MR1117232} ou la monographie de Liggett~\cite{MR1717346}.

Nous avons choisi de mettre l'accent sur la preuve du théorème de forme asymptotique et des propriétés de stationnarit\'e de de sous-additivit\'e du temps d'atteinte essentiel $\sigma$. Nous admettons dans les parties~\ref{sigma},~\ref{controlediff} et~\ref{forme} les contrôles uniformes de la proposition~\ref{propuniforme}: ils seront établis par des arguments de redémarrage dans la section~\ref{restartestimees}, qui est totalement indépendante du reste de l'article. Un appendice est consacré à la preuve d'un théorème ergodique presque sous-additif adapté à nos besoins.

\section{Propriétés des transformations $\tilde{\theta}_x$}
\label{sigma}

\subsection{Premières propriétés}

On commence par vérifier que $K(x)$ est presque sûrement fini, et donc que le temps d'atteinte essentiel $\sigma(x)$ correctement défini. Pour cela, on montre que sous $\P_\lambda$, la loi de $K(x)$ est sous-géométrique:
\begin{lemme}
\label{Kgeom}
$\displaystyle \forall A\subset\Zd\quad\forall x\in\Zd\quad \forall\lambda\in\Lambda\quad \forall n\in\N\quad\P_\lambda(K(x)>n)\le(1-\rho)^n$.
\end{lemme}

\begin{proof}
Rappelons que $\rho$ est la constante apparaissant dans~(\ref{uniftau}). Soit $\lambda \in \Lambda$ et
$n \in \N$. En utilisant la propriété de Markov forte au temps $u_{n+1}$, on obtient:
\begin{eqnarray*}
\P_\lambda(K(x)>n+1)
& = & \P_\lambda(u_{n+2}(x)<+\infty) \\
&  \le  & \P_\lambda(u_{n+1}(x)<+\infty,v_{n+1}(x)<+\infty) \\
& \le & \P_\lambda(u_{n+1}(x)<+\infty,\tau^x \circ \theta_{u_{n+1}(x)}<+\infty) \\
 & \le & \P_\lambda(u_{n+1}(x)<+\infty)\P_\lambda(\tau^x <+\infty)\\
& \le & \P_\lambda(u_{n+1}(x)<+\infty) (1-\rho)=\P_\lambda(K(x)>n)(1-\rho),
\end{eqnarray*}
ce qui prouve le lemme.
\finpreuve \end{proof}


\begin{lemme}
Soit $\lambda\in\Lambda$. Sous $\P_\lambda$, on a, presque sûrement, pour tout $x$ dans $\Zd$, 
\begin{equation}
\label{bleue}
(K(x)=k)\text{ et }(\tau^0=+\infty) \;  \Longleftrightarrow \; (u_k(x)<+\infty \text{ et } v_k(x)=+\infty),
\end{equation}
\end{lemme}

\begin{proof}
Fixons $\lambda \in \Lambda$. Le lemme~\ref{Kgeom} assure que $K(x)$ est $\P_{\lambda}$ presque sûrement fini. Plaçons-nous dans le cas où le processus de contact issu de $0$ survit, c'est à dire sur $\{\tau^0=+\infty\}$. Soit $k \in \N$: en appliquant la propriété de Markov forte au temps d'arrêt $v_k(x)$, on obtient
\begin{eqnarray*}
& & \P_{\lambda}(\tau^0=+\infty, v_k(x)<+\infty, u_{k+1}(x)=+\infty| \mathcal F_{v_k(x)})\\& = & \1_{\{v_k(x)<+\infty\}}  \P_\lambda(\tau^{\centerdot}=+\infty, t^{\centerdot}(x)=+\infty) \circ \xi^0_{v_k(x)}.
\end{eqnarray*}
Maintenant, soit $B$ une partie finie non vide de $\Zd$: avec l'estim\'ee~(\ref{retouche}), on obtient
\begin{eqnarray*}
\P_\lambda ( \tau^B=+\infty, t^B(x)=+\infty) 
& \le & \sum_{y \in B}\P_\lambda(\tau^y=+\infty, t^y(x)=+\infty) \\
& \le & \sum_{y \in B}\P_{\T{y}{\lambda}}(\tau^0=+\infty, t^0(x-y)=+\infty)=0.
\end{eqnarray*}
Donc $\P_{\lambda}(\tau^0=+\infty, v_k(x)<+\infty, u_{k+1}(x)=+\infty)=0$, ce qui implique que sous $\P_\lambda$, 
\begin{equation}
(K(x)=k)\text{ et }(\tau^0=+\infty) \;  \Longrightarrow \; (u_k(x)<+\infty \text{ et } v_k(x)=+\infty),
\end{equation}
autrement dit le procédé de redémarrage s'arrête parce qu'on a trouvé un instant -- $u_{K(x)}$ -- où la descendance de $x$ est infinie.

La réciproque vient alors de la propriété trajectorielle~(\ref{semigroupe}).
\finpreuve \end{proof}
La construction que nous avons présentée ici est très semblable au procédé de redémarrage de Durrett et Griffeath~\cite{MR656515}. La différence essentielle est que, dans leur article, il s'agit de trouver un point proche de $x$ qui a une descendance infinie, alors qu'ici il faut toucher exactement $x$. 
Ainsi, dès lors que l'on sait que le processus partant de $x$ redémarre, on pourra décrire précisément la loi après redémarrage et ainsi  mettre en place des transformations laissant $\Pbarre$ invariante.

\begin{lemme} 
\label{magic}
Soit $x \in \Zd \backslash \{0\}$, $A$ dans la tribu engendrée par $\sigma(x)$, et $B\in \mathcal F$. Alors
$$\forall \lambda \in \Lambda \quad \Pbarre_\lambda(A \cap (\tilde{\theta}_x)^{-1}(B))=\Pbarre_\lambda(A) \Pbarre_{\T{x}{\lambda}}(B).$$
\end{lemme}

\begin{proof} Il suffit de montrer que pour tout $k\in\N^*$, on a
$$\Pbarre_\lambda(A \cap (\tilde{\theta}_x)^{-1}(B) \cap \{K(x)=k\})=\Pbarre_\lambda(A\cap \{K(x)=k\}) \Pbarre_{\T{x}{\lambda}}(B).$$ Comme $A$ est dans la tribu engendrée par $\sigma(x)$, il existe un borélien $A'$ de $\R$ tel que $A=\{\sigma(x) \in A'\}$. 
Le temps d'attente essentiel $\sigma(x)$ n'est pas un temps d'arrêt, mais  on peut utiliser les temps d'arrêt de la construction séquentielle.  
\begin{eqnarray}
&&\P_\lambda(\{\tau^0=+\infty\} \cap A \cap (\tilde{\theta}_x)^{-1}(B) \cap \{K(x)=k\}) \nonumber \\
& {=}  & \P_\lambda(\tau^0=+\infty, \; \sigma(x) \in A', \;   T_x \circ \theta_{\sigma(x)} \in B, \; u_k(x)<+\infty, \; v_k=+\infty) \label{un}\\
& =& \P_\lambda(u_k(x)<+\infty,  u_k(x) \in A', \; \; \tau^x\circ \theta_{u_k(x)}=+\infty, \;T_x \circ \theta_{u_k(x)}\in B) \label{deux}\\
& {=} & \P_\lambda( u_k(x) \in A', \; u_{k}(x)<+\infty) \P_{\lambda}(\tau^x=+\infty, \; T_x \in  B) \label{trois}\\
& {=} &\P_\lambda( u_k(x) \in A', \; u_{k}(x)<+\infty)\P_{\T{x}{\lambda}}(\{\tau^0=+\infty\} \cap B) \label{quatre}.
\end{eqnarray}
Pour (\ref{un}), on utilise l'équivalence~(\ref{bleue}). Pour l'égalité (\ref{deux}), on remarque que pour tout temps d'arrêt $T$,
\begin{equation}
\label{violette}
\{T<+\infty, \; x \in \xi_T^0, \; \tau^0 \circ T_x \circ \theta_T=+\infty\}\subset \{\tau^0=+\infty\}.
\end{equation}
L'égalité (\ref{trois}) résulte de la propriété de Markov forte appliquée au temps d'arrêt $u_{k}$, tandis que (\ref{quatre}) découle de la propriété de translation spatiale~(\ref{translationspatiale}). En divisant l'identité par $\P_{\lambda}(\tau^0=+\infty)$, on obtient une identité de la forme
$$\Pbarre_\lambda(A \cap (\tilde{\theta}_x)^{-1}(B) \cap \{K(x)=k\})=\psi(x,\lambda,k,A) \Pbarre_{\T{x}{\lambda}}(B),$$
et on identifie la valeur de $\psi(x,\lambda,k,A)$ en prenant $B=\Omega$.
\finpreuve \end{proof}

\begin{coro}
\label{invariancePbarre}
Soient $x$ et $y$ dans $\Zd$ et $\lambda\in\Lambda$. On suppose que $x \neq 0$.
\begin{itemize}
\item La translation $\tilde \theta_x$ laisse $\Pbarre$ invariante.
\item Sous $\Pbarre_\lambda$, $\sigma(y)\circ\tilde{\theta}_x$ est indépendant de $\sigma(x)$. De plus, la loi de $\sigma(y)\circ\tilde{\theta}_x$ sous $\Pbarre_\lambda$ est la même que la loi 
de $\sigma(y)$ sous $\Pbarre_{\T{x}{\lambda}}$.
\item  Les variables $(\sigma(x) \circ (\tilde \theta_{x})^j)_{j \ge 0}$ sont indépendantes sous~$\Pbarre_\lambda$. 
\end{itemize}
\end{coro}

\begin{proof}
Pour montrer le premier point, il suffit d'appliquer le lemme précédent avec $A=\Omega$, puis d'intégrer en $\lambda$ en utilisant la stationnarité de $\nu$.

Pour le second point, on considère  $A',B'$ deux boréliens de $\R$ et
on applique le lemme~\ref{magic} avec $A=\{\sigma(x)\in A'\}$ et $B=\{\sigma(y) \circ \tilde{\theta}_x\in B'\}$.

Soit enfin $n \ge 1$ et $A_0,A_1, \dots, A_n$ des boréliens de $\R$. On a:
\begin{eqnarray*}
&& \Pbarre_\lambda(\sigma(x) \in A_0, \sigma(x) \circ \tilde \theta_{x} \in A_1, \dots , \sigma(x) \circ (\tilde \theta_{x})^n \in A_n) \\
& = & \Pbarre_\lambda(\sigma(x) \in A_0, (\sigma(x), 
\dots, \sigma(x) \circ (\tilde \theta_{x})^{n-1})\circ \tilde \theta_{x} \in A_1\times  \dots \times A_n) \\
& {=} & \Pbarre_\lambda(\sigma(x) \in A_0) \Pbarre_{\T{x}{\lambda}}(\sigma(x) \in A_1, \sigma(x) \circ \tilde \theta_{x} \in A_2, \dots , \sigma(x) \circ (\tilde \theta_{x})^{n-1} \in A_n),
\end{eqnarray*}
où la dernière égalité vient du lemme~\ref{magic}.
Par récurrence, on obtient
$$\Pbarre_\lambda \left( \bigcap_{0 \le j \le n} \{ \sigma(x) \circ (\tilde \theta_{x})^j \in A_j\}\right)=\prod_{0 \le j \le n} \Pbarre_{\T{jx}{\lambda}} \left(\sigma(x)\in A_j\right),$$
ce qui conclut la preuve du lemme.
\finpreuve \end{proof}

\subsection{Ergodicité des transformations $\tilde{\theta}_x$}

Pour montrer le théorème~\ref{systemeergodique}, 
il est naturel de chercher à estimer, pour des événements $A$ et $B$, 
comment évolue avec $m$ la dépendance entre les  événements  $A$ et $\tilde{\theta}_x^{-m}(B)$. 
Si $m \ge 1$, l'opérateur $\tilde{\theta}_x^m$ réalise globalement une translation spatiale de vecteur $mx$ et une translation temporelle de vecteur $S_m(x)$:
 \begin{eqnarray*}
 \tilde{\theta}_x^m & = & T_{mx} \circ \theta_{S_m(x)}, \\
 \text{avec } S_m(x)  & = & \sum_{j=0}^{m-1} \sigma(x) \circ \tilde{\theta}_x^j.
 \end{eqnarray*}
On commence par établir un lemme dans l'esprit du lemme~\ref{magic}:
\begin{lemme} 
\label{supermagic}
Soit $t>0$, soit $A \in \mathcal F_t$ et soit $B\in \mathcal F$. 
Alors pour tout $x \in \Zd$, pour tout $\lambda \in \lambda$, pour tout $m \ge1$, 
$$\Pbarre_\lambda(A \cap \{t\le S_m(x)\} \cap (\tilde{\theta}_x^m)^{-1}(B))=\Pbarre_\lambda(A\cap \{t\le S_m(x)\}) \Pbarre_{\T{mx}{\lambda}}(B).$$
\end{lemme}

\begin{proof} Posons $\overline{K}_m(x)=(K(x),K(x)\circ\tilde{\theta}_x,\dots,K(x)\circ\tilde{\theta}_x^{m-1})$. Il suffit de montrer que pour tout $k=(k_0,\dots,k_{m-1})\in(\N^*)^m$, on a
\begin{eqnarray*}
& & \Pbarre_\lambda(A, \;t\le S_m(x),\;\tilde{\theta}_x^{-m}(B),\;\overline{K}_m(x)=k) \\
& = & \Pbarre_\lambda(A,\; t\le S_m(x),\; \overline{K}_m(x)=k) \Pbarre_{\T{mx}{\lambda}}(B).
\end{eqnarray*}

\begin{figure}[h!]
\includegraphics[scale=0.5]{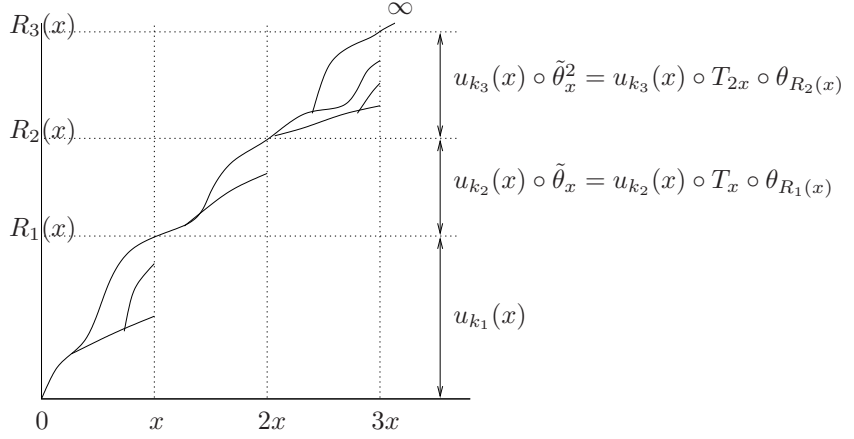}
\put(-165,0){$0$}%
\put(-122,0){$x$}%
\put(-81,0){$2x$}%
\put(-38,0){$3x$}%
\put(-7,40){$u_{k_1}(x)$}
\put(-175,73){$R_1(x)$}%
\put(-7,91){$u_{k_2}(x)\circ \tilde{\theta}_x=u_{k_2}(x)\circ T_x\circ\theta_{R_1(x)}$}
\put(-175,110){$R_2(x)$}%
\put(-7,130){$u_{k_3}(x)\circ \tilde{\theta}_x^2=u_{k_3}(x)\circ T_{2x}\circ\theta_{R_2(x)}$}
\put(-175,150){$R_3(x)$}%
\put(-32,157){$\infty$}
\caption{Un exemple avec $k_1=3$, $k_2=2$ et $k_3=4$.}
\label{uneautrefigure} 
\end{figure}

Soit $k\in (\N^*)^m$. On pose $R_0(x)=0$ et, pour $l \le m-2$, $R_{l+1}(x)=R_l +u_{k_l}(x) \circ \theta_{R_l(x)}$.
Grâce à la remarque~(\ref{violette}), on a l'égalité entre les deux événements suivants:
$$
\left\{
\begin{array}{c}
\tau^0=+\infty ,\\
\overline{K}_m(x)=k
\end{array}
\right\}
=
\left\{
\begin{array}{c}
 u_{k_1}(x)<+\infty, \; 
u_{k_2}(x) \circ T_x\circ \theta_{R_1(x)}<+\infty, \;   \dots , \\
  u_{k_m}(x) \circ T_{(m-1)x} \circ \theta_{R_{m-1}(x)}<+\infty,\\
 \tau^0 \circ T_{mx}\circ \theta_{R_m(x)}=+\infty
\end{array}
\right\}
$$
ainsi que, sur cet événement,  l'identité $S_m(x)=R_m(x)$. Ainsi,
\begin{eqnarray*}
\P_\lambda
\left(
\begin{array}{c}
\tau^0=+\infty,\;  A,\\
 t\le S_m(x),\\
\overline{K}_m(x)=k, \; \tilde{\theta}_x^{-m}(B)
\end{array}
\right)
& = & \P_\lambda
\left(
\begin{array}{c}
A, \; u_{k_1}(x)<+\infty, \;\\
 u_{k_2}(x) \circ T_x\circ \theta_{R_1(x)}<+\infty, \dots, \\ 
u_{k_m}(x) \circ T_{(m-1)x} \circ \theta_{R_{m-1}(x)}<+\infty, \;\\ t\le R_m(x), \; 
\tau^0 \circ T_{mx}\circ \theta_{R_m(x)}=+\infty, \;\\ T_{mx} \circ \theta_{R_m(x)} \in B
\end{array}
\right).
\end{eqnarray*}
Par construction, $R_m(x)$ est un temps d'arrêt et l'événement
$$ A \cap \{u_{k_1}(x)<+\infty\}  \cap \dots \cap \{ u_{k_m}(x) \circ T_{(m-1)x} \circ \theta_{R_{m-1}(x)}<+\infty\} \cap \{t\le R_m(x)\}$$
est dans $\mathcal F_{R_m(x)}$. En utilisant la propriété de Markov forte et la propriété de translation spatiale~(\ref{translationspatiale}), il vient donc:
\begin{eqnarray*}
\P_\lambda
\left(
\begin{array}{c}
\tau^0=+\infty,\;  A,\\
 t\le S_m(x),\\
\overline{K}_m(x)=k, \; \tilde{\theta}_x^{-m}(B)
\end{array}
\right)
& = & \P_\lambda
\left(
\begin{array}{c}
A, \; u_{k_1}(x)<+\infty, \; \\
u_{k_2}(x) \circ T_x\circ \theta_{u_{k_1}(x)}<+\infty, \dots \\ 
u_{k_m}(x) \circ T_{(m-1)x} \circ \theta_{R_{m-1}(x)}<+\infty, \;\\ t\le R_m(x) 
\end{array}
\right) \\
&& \quad \times \P_{\T{mx}{\lambda}}(\{\tau=+\infty\} \cap B ).
\end{eqnarray*}
En divisant l'identité par $\P_{\lambda}(\tau=+\infty)$, on obtient une identité de la forme
$$\Pbarre_\lambda(A,\; t\le S_m(x),\; \tilde{\theta}_x^{-m}(B),\; \overline{K}_m(x)=k)=\psi(x,\lambda,k,m,A) \Pbarre_{\T{mx}{\lambda}}(B),$$
et on identifie la valeur de $\psi(x,\lambda,k,m,A)$ en prenant $B=\Omega$.
\finpreuve \end{proof}

Ainsi, on peut maintenant énoncer une propriété de mélange.
\begin{lemme}
\label{dominationsm}
Soit $t>0$ et $q>1$. Il existe une constante $A(t,q)$ telle que pour tout $x \in \Zd \backslash \{0\}$, pour tout
 $A \in \mathcal F_t$,  pour tout $B\in \mathcal F$, $\lambda \in \Lambda$ et  tout ${\ell} \ge1$, on a 
$$|\Pbarre_\lambda(A  \cap (\tilde{\theta}_x^{\ell})^{-1}(B))-\Pbarre_\lambda(A) \Pbarre_{\T{{\ell}x}{\lambda}}(B)|\le A(t,q)q^{-{\ell}}.$$
\end{lemme}

\begin{proof}
Soit ${\ell} \ge 1$ quelconque. Avec le lemme~\ref{supermagic}, on a
\begin{eqnarray*}
 & & |\Pbarre_{\lambda}(A\cap \tilde{\theta}_x^{-\ell}(B))-\Pbarre_{\lambda}(A) \Pbarre_{\lambda}(\tilde{\theta}_x^{-\ell}(B))|\\
& \le &  |\Pbarre_{\lambda}(t\le S_{\ell}(x),A\cap \tilde{\theta}_x^{-\ell}(B))-\Pbarre_{\lambda}(t\le S_{\ell}(x),A) \Pbarre_{\lambda}(\tilde{\theta}_x^{-\ell}(B))|\\ & & \quad\quad +2\Pbarre_{\lambda}(t>S_{\ell}(x))\\
& = & 2\Pbarre_{\lambda}(t>S_{\ell}(x)).
\end{eqnarray*}
Maintenant, fixons nous $\alpha>0$.\\
Avec l'inégalité de Markov, on a $\Pbarre_{\lambda}(S_{\ell}(x) \le t)  \le  \exp(\alpha t) \Ebarre_\lambda(\exp(-\alpha S_{\ell}(x)))$. En utilisant
 les deux derniers points du corollaire~\ref{invariancePbarre}, il vient
\begin{eqnarray*}
 \Ebarre_\lambda(\exp(-\alpha S_{\ell}(x))) 
& \le &\Ebarre_\lambda \left( \exp\left( -\alpha \sum_{j=0}^{{\ell}-1} \sigma(x) \circ \tilde{\theta}_x^j \right) \right) \\
& \le &  \prod_{j=0}^{{\ell}-1}\Ebarre_\lambda \left( \exp(- \alpha \sigma(x) \circ \tilde{\theta}_x^j) \right) 
 =   \prod_{j=0}^{{\ell}-1}\Ebarre_{\T{jx}{\lambda}} ( \exp(- \alpha \sigma(x)).
\end{eqnarray*}
Il nous reste donc à prouver l'existence d'un $\alpha>0$ tel que pour tout $\lambda \in \Lambda$,
$$\Ebarre_{{\lambda}} ( \exp(- \alpha \sigma(x)) \le q^{-1}.$$
Soit $\rho$ la constante donnée dans l'inégalité~(\ref{uniftau}). 
\begin{eqnarray*}
\Ebarre_\lambda ( \exp(- \alpha \sigma(x))) & \le & \frac1\rho \E_\lambda(\exp(- \alpha \sigma(x))) \le \frac1\rho \E_{\lambda_{\max}}(\exp(- \alpha \sigma(x)))\le\frac1{\rho}\frac{2d\lambda_{max}}{\alpha+2d\lambda_{\max}},
\end{eqnarray*}
car $\sigma(x)\ge t(x)$, qui lui-même domine stochastiquement une variable exmonentielle de paramètre $2d\lambda_{\max}$. Ceci donne bien l'inégalité voulue si l'on prend $\alpha$ assez grand.
\finpreuve
\end{proof}

On a maintenant le matériel nécessaire pour passer à la preuve de l'ergodicité des systèmes $(\Omega,\mathcal{F},\Pbarre,\tilde{\theta}_x)$.

\begin{proof}[\debutpreuve du théorème~\ref{systemeergodique}]
On a déjà vu dans le corollaire~\ref{invariancePbarre} que, pour tout $x \in \Zd$, la probabilité $\Pbarre$ est invariante sous $\tilde{\theta}_x$. Pour la preuve de l'ergodicité, on a besoin de complexifier l'espace afin de pouvoir regarder conjointement un environnement aléatoire et un processus de contact aléatoire.

On pose ainsi $\tilde{\Omega}=\Lambda\times\Omega$, que l'on munit de la tribu $\tilde{\mathcal{F}}=\bor[\Lambda]\otimes\mathcal{F}$
et on définit une mesure de probabilité $\Qbarre$ sur $\tilde{\mathcal{F}}$ par
$$\forall (A,B)\in \bor[\Lambda]\times\mathcal{F}\quad \Qbarre(A\times B)=\int_{\Lambda} \1_A(\lambda)\Pbarre_{\lambda}(B)\ d\nu(\lambda).$$
Définissons la transformation $\tilde{\Theta}_x$ sur $\tilde{\Omega}$ en posant
$\tilde{\Theta}_x(\lambda,\omega)=(x.\lambda,\tilde{\theta}_x(\omega))$. Il est facile de voir que la transformation
 $\tilde{\Theta}_x$ laisse $\Qbarre$ invariant. En effet, pour
$(A,B)\in \bor[\Lambda]\times\mathcal{F}$, en utilisant le lemme~\ref{magic}, on a
\begin{eqnarray*}
\Qbarre(\tilde{\Theta}_x(\lambda, \omega)\in A\times B) & = & \Qbarre(x.\lambda \in A, \; \tilde{\theta}_x(\omega)\in B)\\
& = & \int_{\Lambda}\1_A(x.\lambda)\Pbarre_{\lambda}(\tilde{\theta}_x(\omega)\in B)\ d\nu(\lambda)\\
& = & \int_{\Lambda}\1_A(x.\lambda)\Pbarre_{x.\lambda}(B) \ d\nu(\lambda)\\
& = & \int_{\Lambda}\1_A(\lambda)\Pbarre_{\lambda}(B) \ d\nu(\lambda)
 =  \Qbarre(A\times B).
\end{eqnarray*}
Remarquons que si $g(\lambda,\omega)=f(\lambda)$, alors $\int g\ d\Qbarre=\int f\ d\nu$. \\
De même, si $g(\lambda,\omega)=f(\omega)$, alors $\int g\ d\Qbarre=\int f\ d\Pbarre$.

Comme $\mathcal{A}=\miniop{}{\cup}{t\ge 0}\mathcal{F}_t$ est une algèbre qui engendre $\mathcal{F}$, il suffit, 
pour montrer que  $\tilde{\theta}_x$ est ergodique, de montrer que
pour tout $A\in\mathcal{A}$, 
\begin{equation}
\label{lacon}
\text{la suite } \frac1{n}\miniop{n-1}{\sum}{k=0}\1_A(\tilde{\theta}_x^k)
\text{ converge dans } L^2(\Pbarre) \text{ vers }\Pbarre(A).
\end{equation}
On peut voir la quantité ci-dessus comme une fonction des deux variables~$(\lambda,\omega)$.
Ainsi, il est équivalent de montrer que la suite de fonctions $(\lambda,\omega)\mapsto \frac1{n}\miniop{n-1}{\sum}{k=0}\1_A(\tilde{\theta}_x^k\omega)$ converge vers $\Pbarre(A)$ dans $L^2(\Qbarre)$.

Soit $A\in\mathcal{A}$ et $t>0$ tel que $A\in\mathcal{F}_t$. On décompose, pour tout $(\omega,\lambda)\in\tilde{\Omega}$, la somme en deux termes:
\begin{eqnarray*}
\frac1{n}\sum_{k=0}^{n-1}\1_A(\tilde{\theta}_x^k\omega)
& =& 
\frac1{n}\sum_{k=0}^{n-1}\left(\1_A(\tilde{\theta}_x^k\omega)-\Pbarre_{kx.\lambda}(A)\right)+\frac1{n}\sum_{k=0}^{n-1}\Pbarre_{kx.\lambda}(A)
\end{eqnarray*}
Si l'on pose $f(\lambda)=\Pbarre_{\lambda}(A)$,
le second terme peut s'écrire
$$\frac1{n}\sum_{k=0}^{n-1}\Pbarre_{kx.\lambda}(A)=\frac1{n}\sum_{k=0}^{n-1} f(kx.\lambda).$$
Comme $\nu$ est ergodique, le théorème ergodique $L^2$ de Von Neumann dit que cette quantité converge dans $L^2(\nu)$ vers $\int fd\nu=\Pbarre(A)$.
En la regardant comme une fonction des deux variables $(\lambda,\omega)$, cela dit aussi que la  quantité converge  dans $L^2(\Qbarre)$ vers $\Pbarre(A)$.

Posons, pour $k\ge 0$, $$Y_k=\1_A(\tilde{\theta}_x^k\omega)-\Pbarre_{\lambda}(\tilde{\theta}_x^{-k}(A))=\1_A(\tilde{\theta}_x^k\omega)-\Pbarre_{kx.\lambda}(A)$$
et $L_n=Y_0+Y_1+\dots+Y_{n-1}$.
Il reste donc à montrer que $L_n/n$ converge vers~$0$ dans $L^2(\Qbarre)$. 
Comme $Y_k=Y_0\circ \tilde{\Theta}_x^k$, le champ $(Y_k)_{k\ge 0}$ est stationnaire. On a donc
\begin{eqnarray*} 
\int L_n^2 \ d\Qbarre & = & \sum_{0\le i,j\le n-1}\int  Y_i Y_j\ d\Qbarre \\
& =& \sum_{i=0}^{n-1} \int Y_i^2\ d\Qbarre+2\sum_{\ell=1}^{n-1} (n-\ell) \int  Y_0 Y_{\ell}\ d\Qbarre\\
& \le & 2n \left(\sum_{\ell=0}^{+\infty} \left|\int  Y_0 Y_{\ell}\ d\Qbarre\right|  \right)\\
& \le & 2n \left(\sum_{\ell=0}^{+\infty} \int_{\Lambda}|\Ebarre_{\lambda}  (Y_0 Y_{\ell})|\ d\nu(\lambda)  \right)\\
& \le & 2n \left(\sum_{\ell=0}^{+\infty} \int_{\Lambda}| \Pbarre_{\lambda}(A\cap \tilde{\theta}_x^{-\ell}(A))-\Pbarre_{\lambda}(A) \Pbarre_{\lambda}(\tilde{\theta}_x^{-\ell}(A))|\ d\nu(\lambda)  \right)\\
& \le & 2n \left(\sum_{\ell=0}^{+\infty} A(t,2)2^{-\ell} \right)=4A(t,2)n,
\end{eqnarray*}
grâce au lemme~\ref{dominationsm}. Ceci termine la preuve de  la convergence~(\ref{lacon}) et du théorème~\ref{systemeergodique}.
\finpreuve\end{proof}


\section{Contrôle de l'écart $\sigma(x)-t(x)$ à environnement fixé}
\label{controlediff}
Dans cette section, nous allons contrôler 
le défaut de sous-additivité de $\sigma$, c'est-à-dire les quantités du type $\sigma(x+y)-\sigma(x)-\sigma(y)\circ \tilde{\theta}_x$ d'une part,
et la différence $\sigma(x)-t(x)$ d'autre part.
Ces résultats seront utilisés pour
appliquer un  théorème ergodique presque sous-additif dans la section~\ref{forme}. Dans les deux cas, on va appliquer une procédure de redémarrage, argument fréquemment utilisé dans l'étude des systèmes de particules: ici, vu la définition de $\sigma$ on devra contrôler
des sommes de variables aléatoires de type $v_i-u_i$ et $u_{i+1}-v_i$,  qui sont de nature assez différente:
\begin{itemize} 
\item La durée de vie $v_i(x) -u_i(x)$ de la descendance de $x$ peut être majorée de manière indépendante de la valeur précise de la configuration au temps $u_i(x)$ puisqu'on ne regarde qu'une descendance issue d'un unique point. Le contrôle de ces contributions est donc assez simple.
\item En revanche, le temps $u_{i+1}(x)-v_i(x)$ nécessaire à la réinfection du point $x$  dépend évidemment de la configuration au temps $v_i(x)$, que l'on peine à contrôler précisément et surtout uniformément en $x$. Cela explique que l'argument de redémarrage utilisé ici soit plus complexe et les estimées obtenues moins bonnes que dans des situations plus classiques où on peut avoir des contrôles exponentiels, comme ce sera le cas dans la section~\ref{restartestimees}. 
\end{itemize}
Comme illustration du premier point, on obtient facilement:

\begin{lemme}
\label{LEMtpsdevie}
Il existe des constantes $A,B>0$ telles que pour tout $\lambda \in \Lambda$,
\begin{equation}
\forall x \in \Zd \quad \forall t>0 \quad \Pbarre_{\lambda}(\exists i<K(x): \; v_i(x)-u_i(x) >t) \le A\exp(-Bt).
\end{equation}
\end{lemme}

\begin{proof}
Soit $F:\Omega \to \R$ une fonction mesurable positive et $x \in \Zd$. Posons
$$\mathcal{L}_x(F)=\sum_{i=0}^{+\infty}\1_{\{u_i(x)<+\infty\}} F\circ \theta_{u_i(x)}.$$
Avec la propriété de Markov et la définition de $K(x)$, on a
\begin{eqnarray*}
\E_{\lambda}[\mathcal{L}_x(F)] & = &\sum_{i=0}^{+\infty}\E_{\lambda}[\1_{\{u_i(x)<+\infty\}}]\E_{\lambda}[ F]= \left(1+\sum_{i=0}^{+\infty}\P_{\lambda}(K(x)>i)\right)\E_{\lambda}[F]\\ 
& =& (1+\E_{\lambda}[K(x)])\E_{\lambda}[F]\le\left( 1+\frac1\rho\right)\E_{\lambda}[F],
\end{eqnarray*}
la dernière inégalité provenant du lemme~\ref{Kgeom}. On choisit $F=\1_{\{t<u_i(x)-v_i(x)<+\infty\}}$, et avec l'inégalité~(\ref{uniftau}), il vient :
\begin{eqnarray*}
\Pbarre_{\lambda}(\exists i<K(x): \; v_i(x)-u_i(x) >t) & \le &
\frac{1}{\rho}\P_{\lambda}(\exists i<K(x): \; v_i(x)-u_i(x) >t) \\
& \le & \frac{1}{\rho} \P_{\lambda}(\mathcal{L}_x(F) \ge 1) 
 \le \frac{1}{\rho} \E_{\lambda}[\mathcal{L}_x(F)]  \\
 & \le & \frac{1}{\rho}\left( 1+\frac1\rho\right) \P_{\lambda}(t<\tau^x<+\infty).
\end{eqnarray*}
On conclut alors avec l'inégalité~(\ref{grosamasfinis}).
\end{proof}

Pour traiter les temps de réinfection de type $u_{i+1}(x)-v_i(x)$, l'idée est de chercher, proche du point $(x,u_i(x))$ -- en coordonnées spatio-temporelles --, un point $(y,t)$, infecté depuis $(0,0)$, de temps de vie infini: on pourra alors, avec l'estimée de croissance au moins linéaire, qu'on peut réinfecter $x$, pas trop longtemps après $v_i(x)$, en partant de ce nouveau point source $(y,t)$. Toute la difficulté consiste à bien contrôler la distance entre $(x,u_i(x))$ et un point source $(y,t)$.
 Si la configuration autour de $(x,u_i(x))$ est "raisonnable",  ce point sera proche de $(x,u_i(x))$, ce qui donnera le contrôle souhaité entre $u_{i+1}(x)$ et $u_i(x)$.

On rappelle que pour tout $x\in\Zd$, $\omega_x$ est la mesure ponctuelle dont le support est formé par les temps des morts possibles au site  $x$, et que $M$ (qu'on peut supposer supérieur ou égal à $1$) et $c$ sont donnés dans les estimées~(\ref{richard}) et~(\ref{retouche}). On note
\begin{equation}
\gamma=3M(1+1/c)>3.
\label{gamma}
\end{equation}
Pour $x,y\in\Zd$ et $t>0$, on dit que le point $y$ a une mauvaise croissance par rapport à $x$ si l'événement suivant est réalisé:
\begin{eqnarray*}
E^y(x,t)& =& \{\omega_y[0, t/2]=0\} \cup \{H^y_{ t}\not\subset y+B_{M t}\} \\
& \cup & \{ t/2<\tau^y<+\infty\} \cup \{\tau^y=+\infty, \; \inf\{s\ge 2 t: \; x\in \xi^y_{s}\}> \gamma t\},
\end{eqnarray*}
Dans notre cadre, on travaille avec le formalisme des mesures ponctuelles de Poisson pour calculer le nombre de points à croissance mauvaise par rapport à $x$ dans une boîte spatio-temporelle autour de $(x,0)$: pour tout $x\in\Zd$, tout $L> 0$ et tout $t>0$, on définit le nombre de points à croissance mauvaise dans une boîte spatio-temporelle:
$$N_L(x,t)=\sum_{y\in x+B_{M t+2}} 
\int_0^L \1_{E^y(x,t)}\circ \theta_s \ d \left( \omega_y+\sum_{e\in\Ed:y\in e}\omega_e+\delta_0 \right)(s).$$
Autrement dit, on compte le nombre de couple $(y,s)$ dans la boîte spatio-temporelle $(x+B_{M t+1})\times[0,L]$ tels que 
\begin{itemize}
\item il se passe quelque chose pour $y$ au temps $s$, soit une mort possible, soit une infection possible (et on rajoute l'instant $0$ pour des raisons techniques);
\item l'événement $E^y(x,t)\circ \theta_s$ est réalisé.
\end{itemize}
Nous allons tout de suite  voir que si la boîte spatio-temporelle ne contient que des points avec une bonne croissance, et que $u_i(x)$ est dans la fenêtre temporelle, alors on contrôle le délai avant la prochaine réinfection $u_{i+1}(x)$:

\begin{lemme}
\label{futfut}
Si $N_L(x,t)\circ\theta_{s}=0$
et $s+ t\le u_i(x)\le s+L$, alors
 $v_i(x)=+\infty$ ou $u_{i+1}(x)-u_i(x)\le \gamma t$.
\end{lemme}

\begin{proof}
Par définition de $u_i(x)$, le point $x$ est infecté depuis $(0,0)$ au temps $u_i(x)$. Comme $s+t \le u_i(x)\le s+L$ et que c'est un temps d'infection possible pour $x$, la non-réalisation de $E^x(x,t) \circ \theta_{u_i(x)}$ assure que $\tau^x\circ\theta_{u_i(x)}=+\infty$ ou 
 $\tau^x\circ\theta_{u_i(x)}\le t/2$.
Si $\tau^x\circ\theta_{u_{i}(x)}=+\infty$, c'est terminé car alors $v_{i}(x)=+\infty$.
Sinon, on peut déjà noter que $v_{i}(x)-u_i(x)\le t/2$.

Par définition, il existe un chemin d'infection $\gamma_i :  [0,u_i(x)]\rightarrow \Z^d$ entre $(0,0)$ et $(x,u_i(x))$, c'est-à-dire tel que $\gamma_i(0)=0$ et $\gamma_i(u_i(x))=x$. 
Considérons la portion du chemin d'infection $\gamma_i$ comprise entre les instants de $u_i(x)- t$ et $u_i(x)$. Notons $x_0=\gamma_i(u_i(x))$: nous allons voir que $x_0 \in x+B_{M t+2}$. En effet, si $x_0 \notin x+B_{M t+2}$, on regarde le prochain instant $t_1$ où $\gamma_i$ entre dans $x+B_{M t+2}$ au point, disons, $x_1$ (remarquons que comme $x_1$ est à la frontière intérieure de $x+B_{M t+2}$, alors $\|x-x_1\|_\infty \ge Mt+1$): c'est un instant de possible infection pour $x_1$, et la non-réalisation de $E^{x_1}(x,t) \circ \theta_{t_1}$ assure que $x$ ne pourra être atteint avant un délai $t$  depuis $(x_1,t_1)$, ce qui contredit le fait que $u_i(x)-t\ge 0$. 

Donc $x_0 \in x+B_{M t+2}$: comme $N_L(x,t)\circ\theta_{s}=0$, tout intervalle de temps de longueur supérieure à $t/2$ inclus dans $[s,s+L]$ au dessus de $x_0$ contient une guérison possible; on en déduit que le premier instant $t_2$ où le chemin $\gamma_i$ partant de $(x_0,u_i(x)- t)$ saute en un autre point $x_2$ vérifie $t_2 \le u_i(x)-t+t/2=u_i(x)-t/2$. Ainsi, au moment où $(x_2,t_2)$ infecte $(x,u_i(x))$, il a vécu un temps au moins $t/2$, et la non-réalisation de $E^{x_2}(x,t) \circ \theta_{t_2}$ assure alors qu'il vit infiniment et que 
$$\inf\{u\ge 2 t: \; x \in \xi^{x_2}_{u}\}\circ\theta_{t_2}\le \gamma t.$$
Ainsi il existe $t_3\in [t_2+2 t,t_2+ \gamma t]$, avec 
$x\in\xi^0_{t_3}$. Comme $v_{i}(x)-u_i(x)\le t/2$, 
On a $$t_3\ge t_2+2 t\ge (u_i(x)- t)+2 t=u_i(x)+ t\ge v_i(x).$$
Finalement, $u_{i+1}(x)-u_i(x)\le t_3-u_i(x)\le t_2-u_i(x)+ \gamma t\le \gamma t$.
\finpreuve
\end{proof}

Il faut maintenant vérifier qu'il est très probable qu'une boîte spatio-temporelle ne contienne que des points ayant une bonne croissance vis-à-vis de $x$:

\begin{lemme}
\label{NL}
Il existe des constantes $A_{\ref{eNL}},B_{\ref{eNL}}>0$  telles que pour tout $\lambda\in \Lambda$,
\begin{equation}
\label{eNL}
\forall L>0 \quad \forall x\in\Zd\quad \forall t>0 \quad \P_{\lambda}(N_L(x,t)\ge 1)\le A_{\theequation}(1+L)\exp(-B_{\theequation}t).
\end{equation}
\end{lemme}

\begin{proof} 
Montrons qu'il existe des constantes positives $A_{\ref{eEE}},B_{\ref{eEE}}$ telles que pour tout $\lambda \in \Lambda$, 
\begin{equation}
\label{eEE}
\forall x\in\Zd\quad  \forall t>0 \quad\forall y\in  x+B_{M t+2}\quad \P_{\lambda}( E^y(x,t))\le A_{\theequation}\exp(-B_{\theequation}t).
\end{equation}
Soit $x \in \Zd$, $t>0$ et $y \in  x+B_{M t+2}$.
Si $\tau^y=+\infty$, il existe $z\in \xi^y_{2 t}$ avec $\tau^z\circ\theta_{2t}=+\infty$.
Ainsi, avec la définition~(\ref{gamma}) de $\gamma$,
\begin{eqnarray*}
& & \{\tau^y=+\infty, \; \inf\{s\ge 2t: \; x\in \xi^y_{s}\}> \gamma t\}\\
 & \subset &\{\xi^y_{2 t}\not\subset y+B_{2M t}\}
\cup \bigcup_{z\in y+B_{2M t}}
\{ t^z(x)\circ \theta_{2t}>(\gamma-2M) t \}\\
 & \subset &\{\xi^y_{2t}\not\subset y+B_{2Mt}\}
  \cup\bigcup_{z\in y+B_{2M t}}\left\{t^z(x) \circ \theta_{2t}>\frac{\|x-z\|}c+Mt-\frac{3}c    \right\}.
\end{eqnarray*}
D'où, avec les inégalités~(\ref{richard}) et~(\ref{retouche}),
\begin{eqnarray*}
& &\P_{\lambda}(\tau^y=+\infty, \;\inf\{s\ge 2 t: \; x\in \xi^y_{s}\}> \gamma t)\\ 
&\le& A\exp(-2BM t)+(1+4Mt)^d A\exp(-B(Mt-3/c)).
\end{eqnarray*}
Le nombre $\omega_y([0, t/2])$  de morts possibles sur le site $y$
entre l'instant $0$ et l'instant $t/2$ suit une loi de Poisson de paramètre~$t/2$, donc pour chacune des $2d$ arêtes touchant $y$, on a 
$$\P_{\lambda}(\omega_y([0, t/2])=0)= \exp(- t/2).$$
Les deux autres termes se contrôlent avec les estimées~(\ref{richard}) et~(\ref{grosamasfinis}), ce qui prouve~(\ref{eEE}).

Fixons maintenant $y\in x+B_{Mt+2}$.
Notons $\displaystyle \beta_y=\omega_y+\sum_{e\in\Ed:y\in e}\omega_e$.
Sous~$\P_{\lambda}$, la mesure $\beta_y$ est une réalisation d'un processus ponctuel d'intensité $2d\lambda_e+1$.
Soit $S_0=0$ et $(S_n)_{n\ge 1}$ la suite croissante des instants donnés par ce processus. 
$$\int_0^L \1_{E^y(x,t)}\circ \theta_s \ d(\beta_y+\delta_0)(s)=\sum_{n=0}^{+\infty}\1_{\{S_n\le L\}}\1_{E^y(x,t)}\circ \theta_{S_n},$$
d'où, avec la propriété de Markov,
\begin{eqnarray*}
&& \E_{\lambda}\left(\int_0^L \1_{E^y(x,t)}\circ \theta_s \ d(\beta_y+\delta_0)(s)\right) \\
& =& \sum_{n=0}^{+\infty}\E_{\lambda}\left( \1_{\{S_n\le L\}}\1_{E^y(x,t)}\circ \theta_{S_n}] \right) 
  = \sum_{n=0}^{+\infty}\E_{\lambda} \left( \1_{\{S_n\le L\}} \right) \P_{\lambda}(E^y(x,t))\\ 
& = & \left(1+ \E_{\lambda}[\beta_y([0,L])] \right) \P_{\lambda}(E^y(x,t)) =  (1+L(2d\lambda_e+1))\P_{\lambda}(E^y(x,t)).
\end{eqnarray*}
En utilisant~(\ref{eEE}) et en remarquant que $\P_\lambda(N_L(x,t) \ge 1) \le \E_\lambda[N_L(x,t)]$, on 
termine la preuve.
\finpreuve
\end{proof}

Afin de pouvoir contr\^oler les $u_{i+1}(x)-u_i(x)$ de proche en proche à l'aide du lemme~\ref{futfut}, on doit amorcer le proc\'ed\'e. 
On va supposer pour cela qu'il existe un point $(u,T)$, atteint depuis $0$, de descendance infinie et proche en espace de $x$:
\begin{lemme}
\label{futfut2}
Pour tout $t,T>0$, pour tout $x \in \Zd$, on a l'inclusion suivante:
\begin{eqnarray}
&&\{\tau^0=+\infty\} \label{cavit}\\
&&\cap\left\{\exists u\in x+B_{Mt+2}, \; \tau_u\circ\theta_T=+\infty, \; u\in\xi^0_T\right\} \label{intermediaire}\\
&&\cap \left\{N_{K(x)\gamma t}(x,t)\circ\theta_{T}=0\right\} 
\label{Ntamere} \\ 
&& \cap \bigcap_{1\le i<K(x)} \{v_i(x)-u_i(x)< t\} \label{pasdegrossaut} \\ 
&& \quad \quad \quad \subset \{\tau^0=+\infty\} \cap \left\{\sigma(x)\le T+K(x)\gamma t \right\}. \label{incluse}
\end{eqnarray}
\end{lemme}


\begin{proof}
Si tous les $u_i(x)$ finis sont plus petits que  $T+t$, il n'y a rien à démontrer puisqu'alors
$\sigma(x)\le T+t\le T+K(x)\gamma t$. 
Soit donc $$i_0=\max\{i: u_i(x)\le T+t\}.$$
Comme $v_{i_0}(x)<+\infty$,  l'événement~(\ref{pasdegrossaut}) assure  que 
$ v_{i_0}(x)-u_{i_0}(x)< t,$ et donc
$v_{i_0}(x)\le T+ 2t.$
 Maintenant, comme $\tau^u=+\infty$, la non-r\'ealisation de $E^u(x,t) \circ \theta_T$ impliqu\'ee par~(\ref{Ntamere})
assure que
$$\inf\{s \ge 2t: \; x \in \xi^{u}_s\}\circ \theta_{T} \le \gamma t,$$
ce qui implique que $u_{i_0+1}(x) \le T+\gamma t$.
En remarquant que, par définition de $i_0$, pour tout $j \ge 1$, 
$$u_{i_0+j}(x+y)\ge T+ {t},$$
on montre alors, par récurrence sur $j\le K(x)-i_0$, en appliquant le lemme~\ref{futfut} avec l'événement $\{N_{K(x)\gamma t}(x, t)\circ\theta_{T}=0\}$, que
$$\forall j\in\{1,\dots, K(x)-i_0\}\quad u_{i_0+j}\le T+j\gamma {t};$$
ceci permet d'obtenir, pour $j=K(x)-i_0$,
$$
\sigma(x)  =  u_{i_0+j}(x)\le T+(K(x)-i_0)\gamma t 
 \le  T+K(x)\gamma t,
$$
et achève de prouver l'inclusion~(\ref{incluse}). 
\end{proof}


\subsection{Contrôle du défaut de sous-additivité}

On utilise la stratégie que l'on vient d'expliquer, autour du point $x+y$. Ici, on bénéficie de l'existence d'un départ infini en un point bien précis $(x+y, \sigma(x)+\sigma(y) \circ \tilde{\theta}_y)$ pour appliquer le lemme~\ref{futfut2}.

\begin{proof}[\debutpreuve du théorème~\ref{presquesousadditif}]
Soit $x,y \in \Zd$, $\lambda \in \Lambda$ et $t>0$. On pose $T=\sigma(x)+\sigma(y)\circ\tilde{\theta}_x$.
\begin{eqnarray*}
&& \Pbarre_{\lambda}( \sigma(x+y)>\sigma(x)+\sigma(y)\circ\tilde{\theta}_x+t) \\
& \le & \Pbarre_{\lambda}\left( K(x+y)> \frac{\sqrt t}{\gamma} \right) +\Pbarre_{\lambda}
\left(
\begin{array}{c}
\tau^0=+\infty, \; K(x+y) \le \frac{\sqrt t}{\gamma}  \\
\sigma(x+y)\ge T + K(x+y)\gamma \sqrt t
\end{array}
\right).
\end{eqnarray*}
Le comportement sous-g\'eom\'etrique de la queue de distribution de $K$ donn\'ee par le lemme~\ref{Kgeom} et le contr\^ole uniforme de la probabilit\'e de survie~(\ref{uniftau}) permettent de contr\^oler le premier terme. Remarquons que si $K(x+y) \le \frac{\sqrt t}{\gamma}$, alors $K(x+y) \gamma \sqrt t \le t$, et donc que 
$$\{N_{K(x+y) \gamma \sqrt t}(x+y,\sqrt t) \ge 1\} \subset \{N_{t}(x+y,\sqrt t) \ge 1\}.$$
On applique alors
le lemme~\ref{futfut2} autour de $x+y$, avec une \'echelle $\sqrt t$, un temps de d\'epart $T=\sigma(x)+\sigma(y)\circ\tilde{\theta}_x$ et un point-source $u=x+y$:
\begin{eqnarray*}
&& \Pbarre_{\lambda} \left(
\begin{array}{c}
\tau^0=+\infty, \; K(x+y) \le \frac{\sqrt t}{\gamma}  \\
\sigma(x+y)\ge T + K(x+y)\gamma \sqrt t
\end{array}
\right) \\
& \le & \Pbarre_{\lambda}(N_t(x+y, \sqrt t)\circ\theta_{T} \ge 1) +\Pbarre_\lambda(\exists i<K(x+y): \; v_i(x+y)-u_i(x+y) >\sqrt t)
\end{eqnarray*}
Comme $N_t(x+y,\sqrt{t})=N_t(0,\sqrt{t})\circ T_x\circ T_y$,
on a, en se souvenant que $T=\sigma(x)+\sigma(y)\circ\tilde{\theta}_x$,
$$N_t(x+y,\sqrt t)\circ \theta_{T}=N_t(0,\sqrt{t})\circ\tilde{\theta}_y\circ\tilde{\theta}_x.$$
Ainsi,
$ \Pbarre_{\lambda}(N_t(x+y,\sqrt t)\circ \theta_T \ge 1)=
\le  \Pbarre_{(x+y).\lambda} (N_t(0,\sqrt t)\ge 1),$
et le lemme~\ref{NL} permet de contrôler ce terme. Finalement, le terme $\Pbarre_{\lambda}(\exists i<K(x+y): \; v_i(x+y)-u_i(x+y) > \sqrt t)$ est trait\'e l'aide du lemme~\ref{LEMtpsdevie}.
\finpreuve
\end{proof}

\begin{coro}
\label{momentsecart}
On pose $r(x,y)=(\sigma(x+y)-(\sigma(x)+\sigma(y)\circ\tilde{\theta}_x))^+.$

Pour tout $p\ge 1$, il existe une constante $M_p$ telle que
\begin{equation} \label{momecart}
\forall \lambda\in\Lambda\quad\forall x,y\in\Zd\quad  \Ebarre_{\lambda}[r(x,y)^p]\le M_p.
\end{equation}
\end{coro}

\begin{proof}
Comme $\Ebarre_{\lambda}[r(x,y)^p]=\int_0^{+\infty} pu^{p-1}\Pbarre_{\lambda}(r(x,y)>u)\ du$, d'après la 
proposition~\ref{presquesousadditif} et l'équation~(\ref{uniftau}), il suffit de prendre $$M_p=\frac{pA_{\ref{epresquesousadditif}}}{\rho}\int_0^{+\infty} u^{p-1}\exp(-B_{\ref{epresquesousadditif}}\sqrt u)\ du$$
pour terminer la preuve.
\finpreuve\end{proof}


\subsection{Contrôle de $\sigma(x)-t(x)$}


Ici encore, on voudrait appliquer un procédé analogue à partir de $(x,t(x))$ mais on n'a pas a priori de départ infini proche de ce point. Nous allons, pour trouver un tel point source, le chercher le long du chemin d'infection entre $(0,0)$ et $(x,t(x))$, ce qui n\'ecessitera des contr\^oles sur une boite spatio-temporelle de hauteur $t(x)$, c'est-\`a-dire de l'ordre de $\|x\|$. Nous allons donc perdre \`a la fois dans la pr\'ecision des estim\'ees et dans leur uniformit\'e en $\|x\|$.

\begin{prop} 
\label{importante}
Il existe des constantes $A_{\ref{eimportante}},B_{\ref{eimportante}},\alpha_{\ref{eimportante}}$ strictement positives telles que pour tout $z>0$, tout $x\in\Zd$, pour tout $ \lambda\in\Lambda$:
\begin{eqnarray}
&& \P_\lambda\left(\tau^0=+\infty,\sigma(x)\ge t(x)+K(x)(\alpha_{\theequation}\ln (1+\|x\|) +z)\right)  \nonumber \\
& \le & A_{\theequation}\exp(-B_{\theequation}z). \label{eimportante}
\end{eqnarray}
\end{prop}

\begin{proof} Pour $x,y\in \Zd$ et $t,L>0$, on note 
\begin{eqnarray*}
\tilde{E}^y(t) & = &\{\tau_y<+\infty, \;\miniop{}{\cup}{s\ge 0}H^x_s\not\subset B_{Mt}\}, \\
\tilde N_L(x,t)  & =  & \sum_{y\in x+B_{Mt+1}}\int_0^L \1_{\tilde E^y(Mt+1)}\circ \theta_s \ d \left(\sum_{e\in\Ed:y\in e}\omega_e \right)(s).
\end{eqnarray*}
A l'aide des estim\'ees~(\ref{uniftau}), (\ref{richard}) et (\ref{grosamasfinis}), on voit facilement qu'il existe des constantes $A,B>0$ telles que
\begin{equation}
\forall \lambda\in\Lambda\quad \forall x \in \Zd \quad \forall t>0 \quad \Pbarre_{\lambda}(\tilde N_L(x,t)\ge 1)\le A (1+L)\exp(-B t).
\label{hauteur}
\end{equation}
Maintenant, on choisit $(u,T)$ sur le chemin d'infection entre $(0,0)$ et $(x,t(x))$, de sorte que
$\tau^u \circ \theta_T=+\infty$ et que $T$ soit le dernier instant satisfaisant ces propri\'et\'es (comme on se place sur l'\'ev\'enement $\{\tau^0=+\infty\}$, un tel $T$ existe toujours).

Montrons maintenant que si $\tilde N_{t(x)}(x,t)=0$, alors $u\in x+B_{Mt+2}$.
En effet, si $\|u-x\|> Mt +2$, considérons le premier point $(u',T')$ du chemin d'infection (apr\`es $(u,T)$) qui est
dans $x+B_{Mt}$: par d\'efinition de $T$, la descendance de $(u',T')$ est finie, mais, pour contenir $x$, de diamètre supérieur ou égal à $Mt$, ce qui implique $\tilde N_{t(x)}(x,t)\ge 1$, et donne l'implication souhait\'ee. 

Sur l'\'ev\'enement $\{\tilde N_{t(x)}(x,t)=0\}$, on pourra donc appliquer le lemme~\ref{futfut2} autour du point $x$, avec une \'echelle 
$$t=\frac{\alpha \log(1+\|x\|)+z}{\gamma}\ge \frac{z}{\gamma},$$
le point $(u,T)$ comme point-source et une fen\^etre temporelle de hauteur $L=K(x)\gamma t$. Ici, $\alpha>0$ est une constante suffisamment grande que l'on choisira par la suite. Comme $T\le t(x)$,
\begin{eqnarray}
&& \Pbarre_\lambda\left(\sigma(x)\ge t(x)+K(x)(\alpha\ln (1+\|x\|) +z)\right)
 =  \Pbarre_\lambda \left(\sigma(x)\ge t(x)+K(x)\gamma t \right)  \nonumber \\
& \le & \Pbarre_\lambda \left(\sigma(x)\ge T+K(x)\gamma t \right) \nonumber \\
& \le & \Pbarre_\lambda \left(\sigma(x)\ge T+K(x)\gamma t, \; \tilde N_{t(x)}(x,t)=0 \right) +\Pbarre_\lambda \left(\tilde N_{t(x)}(x,t)\ge1 \right) \nonumber \\
& \le & \Pbarre_{\lambda}(N_{K(x)\gamma t}(x, t)\circ\theta_{T} \ge 1) +\Pbarre_\lambda(\exists i<K(x): \; v_i(x)-u_i(x) >t) \nonumber \\
& & \quad + \Pbarre_\lambda (\tilde N_{t(x)}(x,t)\ge1 ). \label{MAJhic}
\end{eqnarray}
Le second terme de (\ref{MAJhic}) est major\'e par le lemme~\ref{LEMtpsdevie}. Contr\^olons le dernier terme de~(\ref{MAJhic}):
\begin{eqnarray*}
 \Pbarre_\lambda \left(\tilde N_{t(x)}(x,t)\ge1 \right) & \le &  \Pbarre_\lambda \left(\tilde N_{\frac{\|x\|}c+z}(x,t)\ge1  \right) +\Pbarre_\lambda \left( t(x)> \frac{\|x\|}c+z \right).
\end{eqnarray*}
L'estim\'ee~(\ref{retouche}) permet de traiter le second terme, et~(\ref{hauteur}) assure que
\begin{eqnarray*}
\Pbarre_\lambda \left(\tilde N_{\frac{\|x\|}c+t}(x,t)\ge1  \right)
& \le & A \left( 1+\frac{\|x\|}c+z \right)\exp(-B t) \\
& \le & A \left( 1+\frac{\|x\|}c+z \right)\exp \left(-\frac{B(\alpha \log(1+\|x\|)+z)}{\gamma} \right) \le  A' \exp(-B'z),
\end{eqnarray*}
d\`es que $\alpha$ est choisie suffisamment grande.

Pour traiter le premier terme de (\ref{MAJhic}), on remarque que
$N_{K(x)\gamma t}(x,t)\circ\theta_{T}\le N_{t(x)+K(x)\gamma t}(x,t).$
Ainsi
\begin{eqnarray*}
& &\Pbarre_{\lambda}(N_{K(x)\gamma t}(x,t)\circ\theta_{T}\ge 1) \\
&\le  & \Pbarre_{\lambda} \left(N_{\frac{\|x\|}c+z+K(x)\gamma t}(x,t)\ge 1\right)+\Pbarre_{\lambda}\left(t(x)\ge\frac{\|x\|}c+z\right).
\end{eqnarray*}
Comme pr\'ec\'edemment, le second terme est contr\^ol\'e \`a l'aide de~(\ref{retouche}), tandis que, en d\'ecoupant suivant les valeurs prises par $K(x)$, en utilisant~(\ref{eNL}), on obtient:
\begin{eqnarray*}
&& \Pbarre_{\lambda} \left(N_{\frac{\|x\|}c+z+K(x)\gamma t}(x,t)\ge 1\right) \\
&\le  & \sum_{k=1}^{+\infty}\sqrt{\Pbarre_{\lambda}(K(x)=k)}\sqrt{\Pbarre_{\lambda}(N_{k\gamma t+\frac{\|x\|}c+z}(x,t)\ge 1)}\\
&\le  & \sum_{k=1}^{+\infty}\sqrt{\Pbarre_{\lambda}(K(x)=k)}\sqrt{A_{\ref{eNL}}(k\gamma t+\frac{\|x\|}c+z)\exp(-B_{\ref{eNL}}t)}\\
& \le & \sqrt { A_{\ref{eNL}}(1+\frac{\|x\|}c)(1+z)(1+\gamma t) }\exp(-\frac{B_{\ref{eNL}}}2 t)\sum_{k=1}^{+\infty}\sqrt{(1+k)\Pbarre_{\lambda}(K(x)=k)}.
\end{eqnarray*}
Le contr\^ole sous-g\'eom\'etrique~(\ref{Kgeom})de la queue de distribution de $K(x)$ assure que la s\'erie est finie, et quitte \`a augmenter $\alpha$ si besoin, on peut trouver $A'',B''>0$ telles que le pr\'efacteur soit major\'e par $A'' \exp(-B''z)$, ce qui termine la preuve.
\finpreuve
\end{proof}

\begin{lemme}
Pour tout $p \ge 1$, il existe une constante $C_{\ref{momv}}(p)$ telle que, pour tout $x \in \Zd$, 
\begin{equation}
\forall \lambda\in\Lambda\quad\Ebarre_{\lambda} (|\sigma(x)-t(x)|^p)  \le  C_{\theequation}(p) (\ln(1+\|x\|))^{p}. \label{momv}
\end{equation}
\end{lemme}
\begin{proof}
Posons $V_x=\frac{\sigma(x)-t(x)}{K(x)}-\alpha\ln(1+\|x\|)$.
D'après la proposition~\ref{importante}, on peut trouver une variable aléatoire 
$W$ avec des moments exponentiels telle que $W$ domine stochastiquement
la loi de $V_x$ sous $\P_{\lambda}$ pour tout $x$ et pour tout $\lambda$.
De même, d'après le lemme~\ref{Kgeom}, la variable aléatoire $K(x)$ est stochastiquement dominée par une variable aléatoire $K'$ géométrique de paramètre~$\rho$.

Posons $v(x)=\sigma(x)-t(x)=K(x)- (\alpha\ln(1+\|x\|)+V_x)$. Soit $p\ge 1$.
D'après l'inégalité de Minkowski,
\begin{eqnarray*}
(\Ebarre_{\lambda} v(x)^p)^{1/p}& \le & \alpha\ln(1+\|x\|)(\Ebarre_{\lambda} (K(x)^p))^{1/p}+\left(\Ebarre_{\lambda} [ K(x)^p V_x^p ]\right)^{1/p}\\
& \le & \alpha\ln(1+\|x\|)(\Ebarre_{\lambda} [K(x)^p])^{1/p}+\left(\Ebarre_{\lambda}[ K(x)^{2p}] \Ebarre_{\lambda}V_x^{2p}\right)^{\frac1{2p}}\\
& \le & \alpha\ln(1+\|x\|)(\E (K'^p))^{1/p}+\left(\E (K'^{2p}) \E [W^{2p}]\right)^{\frac1{2p}},
\end{eqnarray*}
ce qui termine la preuve. 
\finpreuve\end{proof}

On montre ensuite que la différence entre $\sigma(x)$ et $t(x)$  est asymptotiquement négligeable devant $\|x\|$:

\begin{coro} 
\label{pareil}
$\Pbarre$ presque sûrement, $\displaystyle 
\lim_{\|x\|\to +\infty} \frac{|\sigma(x)-t(x)|}{\|x\|}=0$.
\end{coro}

\begin{proof} 
Soit $p> d$: d'après l'équation~(\ref{momv}), on a
\begin{eqnarray*}
\sum_{x\in\Zd} \Ebarre \frac{|\sigma(x)-t(x)|^p}{(1+\|x\|)^p} & \le & C_{\ref{momv}}(p)  \sum_{x\in\Zd} \frac{(\ln(1+\|x\|))^{p}}{(1+\|x\|)^p}<+\infty.
\end{eqnarray*}
Ainsi, presque sûrement, $(\frac{|\sigma(x)-t(x)|}{(1+\|x\|)})_{x \in \Zd}$ est dans $\ell^p(\Zd)$, donc en particulier tend vers zéro.
\finpreuve \end{proof}

\begin{coro}
\label{sigmaa}
Il existe des constantes positives $A_{\ref{esigmaa}},B_{\ref{esigmaa}},C_{\ref{esigmaa}}$ telles que pour tout $\lambda \in \Lambda$,
pour tout $x\in\Zd$, 
\begin{equation}
\forall t>0 \quad \Pbarre_\lambda\left(\sigma(x)\ge C_{\theequation}\|x\|+t\right) \le A_{\theequation}\exp(-B_{\theequation}\sqrt{t}).
\label{esigmaa}
\end{equation}
\end{coro}

\begin{proof}
On considère la constante $\alpha=\alpha_{\ref{eimportante}}$ donnée dans la proposition~\ref{importante}, et on remarque que si
$K(x)\le\frac1{2\alpha} \sqrt{\|x\|+t/2}$ et $z=\alpha \sqrt{\|x\|+t/2}$,
 alors, en utilisant que $\ln(1+u)\le \sqrt u$, on obtient que
$$K(x)[\alpha \ln(1+\|x\|)+z]\le 2z K(x)\le \|x\|+t/2.$$
Ainsi, avec les estimées~(\ref{retouche}) et~(\ref{uniftau}),
\begin{eqnarray*}
&& \Pbarre_\lambda\left(\sigma(x)> \left(\frac1{c}+1\right)\|x\|+t\right) \\
&\le  &\frac1\rho\P_\lambda\left(\tau^0=+\infty, \; \sigma(x)> \left(\frac1{c}+1\right)\|x\|+t\right)\\
& \le & \frac1\rho \P_\lambda \left( \tau^0=+\infty, \; t(x)\ge\frac{\|x\|}{c}+t/2 \right) + \frac1\rho \P_\lambda \left( K(x)>\frac1{2\alpha} \sqrt{\|x\|+t/2}\right)\\
& & \quad \quad \quad + \frac1\rho\P_\lambda\left(\sigma(x)> t(x)+K(x)(\alpha\ln (1+\|x\|) +\alpha \sqrt{\|x\|+t/2})\right).
\end{eqnarray*}
Le premier terme est contrôlé par l'inégalité~(\ref{retouche}). Le deuxième terme est contrôlé par la queue sous-exponentielle de $K(x)$ du lemme~\ref{Kgeom}, et le dernier par la proposition~\ref{importante}.
\finpreuve\end{proof}

On peut alors donner les conditions d'intégrabilité pour $\sigma$ évoquées dans l'introduction de cette section.

\begin{coro}
\label{propmoments}
Pour tout $p\ge 1$, il existe une constante $C_{\ref{moms}}(p)$ telle que 
\begin{equation}
\label{moms}
\forall\lambda\in\Lambda\quad \forall x\in\Zd\quad \Ebarre_{\lambda}[\sigma(x)^p]  \le  C_{\theequation}(p) (1+\|x\|)^{p}.
\end{equation}
\end{coro}

\begin{proof}
Avec l'inégalité de Minkowski, on a 
$$(\Ebarre_{\lambda}[\sigma(x)^p])^{1/p}\le C_{\ref{esigmaa}}\|x\|+(\Ebarre_{\lambda}[((\sigma(x)-C_{\ref{esigmaa}}\|x\|)^+)^p])^{1/p}.$$
Cependant 
\begin{eqnarray*}
\Ebarre_{\lambda}[((\sigma(x)-C_{\ref{esigmaa}}\|x\|)^+)^p]&=& 
\int_0^{+\infty} pu^{p-1}\Pbarre_{\lambda}(\sigma(x)-C_{\ref{esigmaa}}\|x\|>u)\ du\\ & \le& pA_{\ref{esigmaa}}\int_0^{+\infty} u^{p-1}\exp(-B_{\ref{esigmaa}}\sqrt u)\ du,
\end{eqnarray*}
d'après le corollaire~\ref{sigmaa}, ce qui achève la preuve.
\finpreuve\end{proof}

\noindent
\textbf{Remarque.}
Dans les arguments de red\'emarrage classique, on cherche \`a obtenir l'existence de moments exponentiels pour une variable al\'eatoire en se basant sur le lemme suivant: si les $(X_n)_{n \in \N}$ sont des variables al\'eatoires ind\'ependantes identiquement distribu\'ees avec des moments exponentiels, si $K$ est ind\'ependant des $(X_n)_{n \in \N}$ et admet lui-aussi des moments exponentiels, alors $\sum_{n=0}^K X_n$ admet des moments exponentiels. Ici, la difficult\'e \`a contr\^oler pr\'ecis\'ement les temps de r\'einfection du type $u_{i+1}-v_i$ nous emp\^eche d'int\'egrer ce cadre: on doit donc utiliser des arguments \emph{ad hoc} donnant des contr\^oles plus faibles.

\section{Théorèmes de forme asymptotique}
\label{forme}
On peut maintenant passer à la preuve du théorème~\ref{thFA}.
Le premier pas consiste à prouver  l'existence de limites pour les quantités
$\frac{\sigma(nx)}n$.

Grâce au corollaire~\ref{momentsecart}, on a quels que soient les entiers $n$ et $p$:
$$\Ebarre[\sigma((n+p)x)]\le \Ebarre[\sigma(nx)]+\Ebarre[\sigma(px)]+M_1.$$
Alors, d'après le lemme de Fekete, $\frac1{n}\Ebarre[\sigma(nx)]$ admet une limite finie lorsque $n$ tend vers l'infini.
Ainsi, un bon candidat pour être la limite est la quantité
 $$\displaystyle \mu(x)=\lim_{n\to +\infty} \frac{{\Ebarre} (\sigma(nx))}{n}.$$

\begin{theorem}
\label{cvdir}
Pour tout $x \in \Zd$, $\Pbarre$ presque sûrement:
$$\lim_{n \to +\infty} \frac{\sigma(nx)}{n}= \lim_{n \to +\infty} \frac{\Ebarre \sigma(nx)}{n}=\mu(x).$$
Cette convergence a de plus lieu dans tous les $L^p(\Pbarre)$, $p\ge 1$.
\end{theorem}

Pour cela, on va avoir besoin d'établir deux théorèmes ergodiques presque sous-additifs:
\begin{theorem}
\label{therg}
Soit $(\Omega,\mathcal{F},\P)$ un espace probabilisé, $(\theta_n)_{n\ge 1}$ une famille de transformations laissant $\P$ invariante. Sur cet espace sont définies une famille  $(f_n)_{n\ge 1}$   de fonctions intégrables, une famille  $(g_n)_{n\ge 1}$   de fonctions positives ou nulles, et une famille $(r_{n,p})_{n,p\ge 1}$ telles que
\begin{equation}
\label{maiscestimportant}
\forall n,p\ge 1\quad f_{n+p}\le f_{n}+f_{p}\circ\theta_n+g_p\circ\theta_n+r_{n,p}.
\end{equation}
On fait les hypothèses suivantes:
\begin{itemize} 
\item  $c=\inf_{n\ge 1} \frac{\E f_n}n>-\infty$ 
\item $g_1$ est intégrable, $g_n/n$ converge presque sûrement et dans $L^1$ vers $0$
\item il existe $\alpha>1$ et 
une suite $(C_p)_{p\ge 1}$ telles que 
$\E[ (r^+_{n,p})^{\alpha}]\le C_p$ pour tous $n,p$ et
$$\sum_{p=1}^{+\infty} \frac{C_p}{p^{\alpha}}<+\infty.$$
\end{itemize}
Alors,  la suite de terme général $\frac1{n}\E f_n$ converge; si on note $\mu$ sa limite, on a
 $$\E[ \miniop{}{\liminf}{n \to +\infty} \frac{f_n}{n}]\ge \mu.$$
Si l'on pose $\underline{f}=\miniop{}{\liminf}{n \to +\infty} \frac{f_n}{n}$, alors $\underline{f}$ est invariante par tous les $\theta_n$.
\end{theorem}

\begin{theorem}
\label{thergdeux}
On se place sous les hypothèses du théorème~\ref{therg}.
On suppose en plus que pour tout $k$,
$$\frac1{n}\left(f_{nk}-\sum_{i=0}^{n-1}f_k\circ(\theta_k)^i\right)^+\to 0\ \P\ \text{p.s.}$$
Alors $f_n/n$ converge presque sûrement vers $\underline{f}$.
\end{theorem}

\begin{proof}
La preuve de ces résultats, inspirée de celle de Liggett~\cite{MR806224}, sera donnée dans une annexe, où on trouvera également un certain nombre de commentaires.
\finpreuve\end{proof}

\begin{proof}[\debutpreuve du théorème~\ref{cvdir}]
On commence par appliquer le théorème~\ref{therg} en posant $f_n=\sigma(nx)$, $\theta_n=\tilde{\theta}_{nx}$, $g_p=0$, et $r_{n,k}=r(nx,kx)$ et en travaillant avec la probabilit\'e $\Pbarre$. On peut prendre $\alpha>1$ quelconque.
Le corollaire~\ref{propmoments} assure l'intégrabilité de $\sigma(x)$ sous $\Pbarre$ et le corollaire~\ref{momentsecart} donne les contrôles de moments nécessaires.

Vérifions maintenant les hypothèses du théorème~\ref{thergdeux}: on voit facilement que
$$t(nkx)\le \sum_{i=0}^{n-1}\sigma(kx)\circ (\tilde{\theta}_{kx})^i,$$
ce qui implique
$$\left(\sigma(nkx)-\sum_{i=0}^{n-1}\sigma(kx)\circ (\tilde{\theta}_{kx})^i\right)^+\le \sigma(nkx)-t(nkx),$$
qui est bien négligeable devant $n$ en l'infini, d'après le corollaire~\ref{pareil}.

Ainsi $\sigma(nx)/n$ converge bien vers une variable aléatoire $\mu(x)$, qui est invariante par $\tilde{\theta}_x$. D'après le théorème~\ref{systemeergodique}, cette variable est en fait une constante, ce qui achève la preuve de la convergence presque sûre.
Pour montrer qu'une suite
converge dans $L^p$, il suffit de montrer qu'elle converge presque
sûrement et qu'elle est bornée dans $L^q$ pour un $q>p$.
Or, le corollaire~\ref{propmoments} montre que $f_n/n$ est bornée
dans tous les $L^q$, ce qui permet de conclure.
\finpreuve\end{proof}

Nous allons maintenant prouver le théorème~\ref{thFA} de forme asymptotique.  On commence par montrer le théorème de forme asymptotique pour le temps d'atteinte essentiel $\sigma$, en procédant selon les étapes décrites ci-dessous, comme dans le schéma classique:
\begin{itemize}
\item On prolonge $\mu$ en une norme sur $\Rd$ dans le lemme~\ref{munorme}.
\item On montre que la convergence directionnelle du théorème~\ref{cvdir} est en fait une convergence uniforme en la direction dans le lemme~\ref{convunifsigma}.
\item Ce lemme implique facilement le résultat de forme asymptotique du lemme~\ref{formesigma}.
\end{itemize}
Pour montrer les résultats du même type pour le temps d'atteinte classique $t$ (lemme~\ref{etpourt}), il suffit alors de contrôler la différence entre $t$ et $\sigma$, ce que nous avons fait dans le lemme~\ref{pareil}. Finalement, le résultat de forme asymptotique pour la zone couplée est démontré dans le lemme~\ref{zoneC}, en introduisant un temps de couplage $t'$ et en contrôlant la différence entre ce temps $t'$ et le temps d'atteinte essentiel $\sigma$.

\begin{lemme}
\label{munorme}
La fonctionnelle $\mu$ se prolonge en une norme sur $\Rd$.
\end{lemme}

\begin{proof}
\underline{Homogénéité en entiers}
On sait que $\mu(x)=\lim \Ebarre\frac{\sigma(nx)}{n}$, et que
$\sigma(nx)$ et $\sigma(-nx)$ ont même loi sous $\Pbarre$, donc $\mu(x)=\mu(-x)$.
\`A l'aide de suites extraites, on prouve alors la propriété d'homogénéité en entiers:
$$\forall k\in\Z\quad\forall x\in\Zd \quad \mu(kx)=|k|\mu(x).$$  

\noindent
\underline{Sous-additivité}
On a 
$$\sigma(nx+ny)\le \sigma(nx)+\sigma(ny)\circ\tilde{\theta}_{nx}+r(nx,ny).$$
Comme $\tilde{\theta}_{nx}$ laisse invariant $\Pbarre$, il vient
\begin{eqnarray*}
\Ebarre\sigma(nx+ny)& \le & \Ebarre\sigma(nx)+\Ebarre\sigma(ny)+\Ebarre r(nx,ny)\\
& & \Ebarre\sigma(nx)+\Ebarre\sigma(ny)+M_1,
\end{eqnarray*}
d'après le corollaire~\ref{momentsecart},
 ce qui entraîne que $$\forall x \in \Zd \quad \forall y \in \Zd \quad \mu(x+y)\le \mu(x)+\mu(y).$$

\noindent
\underline{Extension à $\Rd$}
D'après le lemme de Fekete, $\mu(x)+M_1=\miniop{}{\inf}{n\ge 1}\frac{\Ebarre{\sigma(nx)+M_1}}n$, d'où  $\mu(x)\le \Ebarre \sigma(x)$.
Le corollaire~\ref{propmoments} donne l'existence de $L>0$ tel que $\Ebarre \sigma(x)\le L\|x\|$ pour tout $x$. Finalement, $\mu(x)\le L\|x\|$ pour tout $x$ dans $\Zd$, ce qui entraîne $|\mu(x)-\mu(y)|\le L\|x-y\|$: on peut alors prolonger 
$\mu$ sur $\Q^d$ par homogénéité, puis sur $\Rd$ par uniforme continuité.

\medskip
\noindent
\underline{Positivité}
Soit $M$ la constante donnée dans la proposition~\ref{propuniforme}. L'estimée (\ref{richard}) nous donne 
\begin{eqnarray*}
\Pbarre \left(\sigma(nx)< \frac{n\|x\|}{2M}\right) & \le &
\Pbarre \left(t(nx)< \frac{n\|x\|}{2M}\right) \le \Pbarre \left(\xi_{\frac{n\|x\|}{M}}\not\subset B_{\frac{n\|x\|}2} \right)\\
& \le &\int \frac{\P_{\lambda}\left(\tau=+\infty,\xi^0_{\frac{n\|x\|}{2M}}\not\subset B_{\frac{n\|x\|}2}\right)}{\P_{\lambda}(\tau=+\infty)}\ d\nu(\lambda)\\
& \le &\frac{A}{\rho}\exp\left(-B\frac{n\|x\|}{2M}\right).
\end{eqnarray*}
On en déduit, avec le lemme de Borel-Cantelli, que  $\mu(x)\ge \frac{1}{2M}\|x\|$.
L'inégalité, établie pour tout $x \in \Zd$, se prolonge par homogénéité et continuité à $\Rd$ tout entier, ce qui entraîne que $\mu$ est une norme. 
\finpreuve \end{proof}

Dans la suite, on pose  $C=2C_{\ref{esigmaa}}$, où $C_{\ref{esigmaa}}$ est donnée dans le corollaire~\ref{sigmaa}.

\begin{lemme}
\label{reglo}
Pour tout $\epsilon>0$, $\Pbarre$ presque sûrement, il existe $R$ tel que
$$\forall x,y \in \Zd \quad (\|x\|\ge R\text{ et }\|x-y\|\le \epsilon \|x\|)\Longrightarrow (|\sigma(x)-\sigma(y)|\le  C\epsilon\|x\|).$$
\end{lemme}

\begin{proof}
Pour $m \in \N$ et $\varepsilon>0$, on définit l'événement 
$$A_m(\epsilon)=\left\{\exists x,y \in \Zd: \; \|x\|=m, \; \|x-y\|\le \epsilon m \text{ et } |\sigma(x)-\sigma(y)|>C\epsilon m \right\}.$$
En remarquant que
$$A_m(\epsilon) \subset  \bigcup_{\substack  {(1-\epsilon) m\le \|x\| \le (1+\epsilon)m \\ \|x-y\|\le \epsilon m}} \left\{\sigma(y-x)\circ\tilde{\theta}_{x}+r(x,y-x)>C\epsilon m\right\},$$
il vient, à l'aide des corollaires~\ref{sigmaa} et~\ref{momentsecart},
\begin{eqnarray*}
\Pbarre_{\lambda} ( A_m(\epsilon) ) & \le &  \sum_{\substack {(1-\epsilon) m\le \|x\| \le (1+\epsilon)m\\ \|z\|\le\epsilon m}} \Pbarre_{{\lambda}}(\sigma(z)\circ\tilde{\theta}_{x}+r(x,z)>C\epsilon m)\\
 & \le &  \sum_{\substack {(1-\epsilon) m\le \|x\| \le (1+\epsilon)m\\ \|z\|\le\epsilon m}}  \Pbarre_{\T{x}{\lambda}}(\sigma(z)>2C\epsilon m/3)\\ & &+ \quad \quad \Pbarre_{{\lambda}}(r(x,y-x)>C\epsilon m/3)\\
& \le & ((1+2\epsilon m)^d((1+2(1+\epsilon) m)^d(A_{\ref{esigmaa}}\exp(-B_{\ref{esigmaa}}\sqrt{C\epsilon m /3})\\ & &\quad \quad +A_{\ref{momecart}}\exp(-B_{\ref{momecart}}\sqrt{C'\epsilon m /3})).
\end{eqnarray*}
grâce au corollaire~\ref{sigmaa} et au théorème~\ref{presquesousadditif}.
On intègre alors cette inégalité par rapport à l'environnement $\lambda$, et le lemme de Borel-Cantelli permet de conclure.
\finpreuve \end{proof}


On peut maintenant montrer que la convergence du théorème~\ref{cvdir}  est uniforme par rapport à la direction.

\begin{lemme} 
\label{convunifsigma}
$\Pbarre$ presque sûrement, $\displaystyle \lim_{\|x\|\to +\infty} \frac{|\sigma(x)-\mu(x)|}{\|x\|}=0$.
\end{lemme}

\begin{proof}
On raisonne par l'absurde et on suppose qu'il existe $\epsilon>0$ tel que l'événement ``$|\sigma(x)-\mu(x)|> \epsilon\|x\|$ pour un  ensemble infini de valeurs de $x$'' a une probabilité strictement positive.
Plaçons nous sur cet événement. Alors, il existe une suite aléatoire   $(y_n)_{n \ge 0}$ de sommets de $\Z^d$ telle que $\|y_n\|_1 \rightarrow + \infty$ et, pour tout $n$, $|\sigma(y_n)-\mu(y_n)| \geq \varepsilon \|y_n\|_1$. 
Quitte à prendre une sous-suite, on peut supposer
$$\frac{y_n}{\|y_n\|_1} \rightarrow z.$$
Approchons $z$ par un point rationnel: considérant  $\varepsilon_1>0 $ (qui sera choisi plus tard), on peut prendre $z' \in\Zd$ tel que 
$$\left\| \frac{z'}{\|z'\|_1} -z\right\|_1\leq \varepsilon_1.$$ 
On peut trouver, pour chaque $y_n$, un point entier sur la ligne $\R z'$ 
qui est suffisamment près de $y_n$:  soit $h_n$ la partie entière de  $\frac{\|y_n\|_1}{\|z'\|_1}$. On a
\begin{eqnarray*} 
\|y_n-h_n.z'\|_1 & \leq & \left\|y_n- \frac{\|y_n\|_1}{\|z'\|_1} z'\right\|_1+ \left| \frac{\|y_n\|_1}{\|z'\|_1}- h_n\right| \|z'\|_1 \\
& \leq &\|y_n\|_1  \left\|\frac{y_n}{\|y_n\|_1}-\frac{z'}{\|z'\|_1}\right\|_1 +\|z'\|_1
\end{eqnarray*}
Prenons $N>0$ suffisamment grand pour  que
$\big(n \geq N\big) \Rightarrow \big(\|\frac{y_n}{\|y_n\|_1}-z\|_1\leq \varepsilon_1\big)$.
Grâce au choix que l'on a fait pour $z'$, on a
$$\big(n \geq N\big) \Rightarrow \big(\left\|\frac{y_n}{\|y_n\|_1}-\frac{z'}{\|z'\|_1}\right\|_1 \leq 2\varepsilon_1\big),$$
et par conséquent,
$ \|y_n-h_n.z'\|_1 \leq 2\varepsilon_1\|y_n\|_1 +\|z'\|_1$.
Ainsi, quitte à augmenter $N$ si besoin, on a pour tout $n\ge N$,
$ \|y_n-h_n.z'\|_1 \leq 3\varepsilon_1\|y_n\|_1$.
Cependant, si on prend $N$ suffisamment grand, le lemme~\ref{reglo}
implique alors que l'on a 
$$\forall n\ge N\quad |\sigma(y_n)-\sigma(h_n.z')| \leq 3C \varepsilon_1\|y_n\|_1.$$ 
Finalement, pour tout $n$ assez grand, on a
\begin{eqnarray*} 
|\sigma(y_n)-\mu(y_n)| & \leq & |\sigma(y_n)-\sigma(h_n.z')|+|\sigma(h_n.z')-\mu(h_n.z')|+ 
|\mu(h_n.z')-\mu(y_n)| \\
& \leq & 3C \varepsilon_1\|y_n\|_1 + h_n \left|\frac{\sigma(h_n.z')}{h_n}-\mu(z')\right| +  L\|h_n.z'-y_n\|_1 \\
& \leq & 3C \varepsilon_1\|y_n\|_1 +(1+\varepsilon_1)\frac{\|y_n\|_1}{\|z'\|_1}\left|\frac{\sigma(h_n.z')}{h_n}-\mu(z')\right|+ 3\varepsilon_1 L\|y_n\|_1.\end{eqnarray*}
Mais la convergence presque sûre dans la direction donnée par $z'$
assure que pour $n$ assez grand
$$\left|\frac{\sigma(h_n.z')}{h_n}-\mu(z')\right| \leq \varepsilon_1.$$
Si l'on prend $\epsilon_1$ assez petit, on obtient que pour $n$ assez grand
$|\sigma(y_n)-\mu(y_n)| < \varepsilon\|y_n\|_1,$
ce qui amène la contradiction.
\finpreuve \end{proof}


On déduit alors de la convergence uniforme du lemme~\ref{convunifsigma} le théorème de forme asymptotique pour la version grossie $\tilde G_t$ de $G_t=\{x\in\Zd: \sigma(x)\le t\}$; on rappelle que $A_\mu$ désigne la boule unité pour la norme $\mu$.

\begin{lemme}
\label{formesigma}
Pour tout $\epsilon>0$, avec probabilité $1$ sous $\Pbarre$, pour tout $t$ suffisamment grand,
$(1-\epsilon)A_{\mu}\subset\frac{\tilde G_t}t\subset (1+\epsilon)A_{\mu}$.
\end{lemme}

\begin{proof}
Montrons par l'absurde que pour $t$ assez grand, on a bien $\frac{G_t}t\subset (1+\epsilon)A_{\mu}$.
Supposons qu'il existe une suite $(t_n)_{n\ge 1}$,
avec $t_n\to +\infty$ et $\frac{G_{t_n}}{t_n}\not\subset (1+\epsilon)A_{\mu}$: il existe donc $x_n$ avec $\sigma(x_n)\le t_n$ et $\mu(x_n)/t_n>1+\epsilon$.
Ainsi $\mu(x_n)/\sigma(x_n)>1+\epsilon$, ce qui contredit la convergence uniforme puisque, comme $\mu(x_n)>t_n(1+\epsilon)$, la suite $(\|x_n\|)_{n\ge 1}$ tend vers l'infini.

Passons à l'inclusion inverse. En raisonnant toujours par l'absurde, on a
une suite $(t_n)_{n\ge 1}$, avec $t_n\to +\infty$ et $(1-\epsilon)A_{\mu}\not\subset\frac{\tilde G_{t_n}}{t_n}$, ce qui veut dire qu'on peut trouver $x_n$ avec $\mu(x_n)\le (1-\epsilon)t_n$, mais $\sigma(x_n)>t_n$.
La suite $(x_n)_{n\ge 1}$ ne peut pas être bornée (c'est à dire ne prendre qu'un nombre fini de valeur) car $t_n$ tend vers l'infini.
Ainsi on a $\frac{\mu(x_n)}{\sigma(x_n)}<1-\epsilon$, ce qui contredit encore une fois la convergence uniforme.
\finpreuve \end{proof}

Grâce au lemme~\ref{pareil}, on récupère alors immédiatement la convergence uniforme pour le temps d'atteinte~$t$ et, par un argument identique à celui du lemme~\ref{formesigma}, le théorème de forme asymptotique pour la version grossie $\tilde H_t$ de $H_t=\{x\in\Zd: t(x)\le t\}$.

\begin{lemme} 
\label{etpourt}
$\Pbarre$ presque sûrement, $\displaystyle \lim_{\|x\|\to +\infty} \frac{t(x)-\mu(x)}{\|x\|}=0,$ \\
et pour tout $\epsilon>0$, avec probabilité $1$ sous $\Pbarre$, pour tout $t$ suffisamment grand,
$$(1-\epsilon)A_{\mu}\subset\frac{\tilde H_t}t\subset (1+\epsilon)A_{\mu}.$$
\end{lemme}

Il nous reste maintenant à prouver le théorème de forme asymptotique pour la zone couplée $\tilde K'_t$, version grossie de $K'_t=\{x\in\Zd: \; \forall s\ge t \quad \xi^0_s(x)=\xi^{\Zd}_s(x)\}$:

\begin{lemme}
\label{zoneC}
Pour tout $\epsilon>0$, avec probabilité $1$ sous $\Pbarre$, pour tout $t$ suffisamment grand,
$(1-\epsilon)A_{\mu}\subset\frac{\tilde K'_t \cap \tilde G_t}t$.
\end{lemme}

\begin{proof}
Comme $t\mapsto K'_t\cap G_t$ est croissant, on se retrouve dans le même schéma de preuve que dans le lemme~\ref{formesigma}. On pose, pour $x \in \Zd$,
$$t'(x)=\inf\{t\ge 0:\;x\in K'_t\cap G_t\}.$$ Il suffit alors de montrer que 
$\Pbarre$ presque sûrement,
$\miniop{}{\lim}{\|x\|\to +\infty} \frac{|t'(x)-\sigma(x)|}{\|x\|}=0$. 
Par définition, $t'(x)\ge \sigma(x)$; ainsi il suffit
de montrer qu'il existe deux constantes $A',B'>0$ telles que
\begin{equation}
\label{unbut}
\forall x\in\Zd\quad \forall s\ge 0\quad\Pbarre( t'(x)-\sigma(x)\ge s)\le A'e^{-B' s}.
\end{equation}

$\bullet$ On commence par remarquer que, pour tout $t \ge0$, $K_{\sigma(x)+t}\supset x+K_t \circ \tilde{\theta}_x$. \\
En effet, soit $z \in x+K_t \circ \tilde{\theta}_x$. 
Considérons d'abord le cas où  $z \not\in \xi_{\sigma(x)+t}^{\Zd}$.
Comme, par additivité~(\ref{additivite}), $\xi_{\sigma(x)+t}^{0}\subset \xi_{\sigma(x)+t}^{\Zd}$,  alors $z \not\in \xi_{\sigma(x)+t}^{0}$, et donc $z\in  K_{\sigma(x)+t}$.

Considérons maintenant le cas où $z\in \xi^{\Zd}_{\sigma(x)+t}$. Comme $\xi^{\Zd}_{\sigma(x)} \subset \xi^{\Zd}_{0}\circ\tilde{\theta}_x$ par additivité, on a $y=z-x \in \xi^{\Zd}_{t}\circ\tilde{\theta}_x$.
Mais comme $y\in K_t \circ \tilde{\theta}_x$, il s'ensuit, par définition de
$K_t$, que $\xi^{0}_{t}(y)\circ\tilde{\theta}_x=\xi^{\Zd}_{t}(y)\circ\tilde{\theta}_x=1$.
Comme $x\in\xi^0_{\sigma(x)}$ et que $y\in  \xi^{0}_{t}\circ\tilde{\theta}_x$, on a $z=x+y\in\xi^0_{\sigma(x)+t}$, d'où $z\in  K_{\sigma(x)+t}$.

\medskip
$\bullet$ Soit $s \ge 0$ fixé. Le point précédent assure que
$$\left(\bigcap_{t\ge s} K_{\sigma(x)+t}\right)  \supset \left( x+\bigcap_{t\ge s}(K_t \circ \tilde{\theta}_x)\right), \text{ et donc }
K'_{\sigma(x)+s}\supset \left(x+(K'_s \circ \tilde{\theta}_x)\right).$$
 Ainsi, en utilisant l'invariance de $\Pbarre$ sous l'action de $\tilde{\theta}_x$, on obtient:
\begin{eqnarray*}
\Pbarre(t'(x)>\sigma(x)+s) & = & \Pbarre(x\notin K'_{\sigma(x)+s}\cap G_{\sigma(x)+s}) \\
& =& \Pbarre(x\notin K'_{\sigma(x)+s})\\
& \le & \Pbarre \left( x\notin (x+K'_s  \circ \tilde{\theta}_{x}) \right) \\
& \le & \Pbarre(0\notin K'_s\circ\tilde{\theta}_{x} )=\Pbarre(0\notin K'_s).
\end{eqnarray*}
L'estim\'ee~(\ref{petitsouscouple}) permet alors de conclure.
\finpreuve \end{proof}

\section{Contrôle uniforme de la croissance en environnement $\lambda$}
\label{restartestimees}

Le but de cette section est d'établir les contrôles uniformes en $\lambda$ annoncés dans la proposition~\ref{propuniforme}. 
Afin de contrôler la croissance du processus de contact, on a besoin de quelques lemmes sur le modèle de Richardson.

\subsection{Quelques lemmes sur le modèle de Richardson}
On appelle modèle de Richardson de paramètre $\lambda$ le processus de Markov  $(\eta_t)_{t\ge 0}$, homogène en temps, qui prend ses valeurs dans $\mathcal{P}(\Zd)$ et dont l'évolution est définie comme suit:  les sites $z$ vides deviennent infectés au taux $\displaystyle \lambda\sum_{\|z-z'\|_1=1} \eta_t(z') $, ces différentes évolutions étant indépendantes les unes des autres. 
Grâce à la construction graphique, on peut, pour tout $\lambda\in\Lambda$,
coupler le processus
de contact en environnement $\lambda$ avec  le modèle de Richardson de paramètre $\lambda_{\max}$, de telle manière que l'espace occupé au temps $t$ par le processus de contact est toujours contenu dans l'espace occupé par le modèle de Richardson.

Le premier lemme, dont nous omettons la preuve, découle aisément de la représentation du modèle de Richardson en terme de percolation de premier passage et d'un comptage de chemins. 

\begin{lemme}
\label{Richardson1}
Pour tout $\lambda>0$, il existe des constantes $A,B>0$ telles que 
$$\forall t\ge 0 \quad \P(\eta_1\not\subset B_t)\le A\exp(-B t).$$
\end{lemme}




\begin{lemme}
\label{Richardson2}
Pour tout $\lambda>0$, il existe des constantes $A,B,M>0$ telles que $$\forall s\ge 0 \quad \P(\exists t\ge 0: \; \eta_t \not\subset B_{Mt+s})\le A\exp(-Bs).$$
\end{lemme} 

\begin{proof}
La représentation en termes de percolation de premier passage du modèle de Richardson
assure l'existence de $A',B',M'>0$ tels que
pour tout $t \ge 0$,
\begin{equation}
\label{voirkesten}
\P(\eta_t\not\subset B_{M't}) \le A'\exp(-B't).
\end{equation}
Pour plus de détails, on pourra se reporter à Kesten~\cite{kesten}.

On commence par contrôler le processus aux temps entiers grâce à cette estimée:
\begin{eqnarray}
 \P(\exists k\in\N: \; \eta_k\not\subset B_{M'k+s/2})
& \le &\P(\exists k\in\N: \; \eta_{k+s/(2M')}\not\subset B_{M'k+s/2}) \nonumber \\
&{\le} & \sum_{k=0}^{+\infty}\P( \eta_{k+s/(2M')}\not\subset B_{M'k+s/2})  \nonumber \\
&  \le &  \frac{A'}{1-\exp(-B')}\exp\left(-\frac{B's}{2M'}\right). \label{olala}
\end{eqnarray}
Contrôlons maintenant les fluctuations entre les temps entiers. 
 Soit $M>M'$:
\begin{eqnarray}
&&  \P(\{\exists t\ge 0:\; \eta_t\not\subset B_{Mt+s}\}\cap \{\forall k \in \N, \; \eta_k\subset B_{M'k+s/2} \}) \nonumber \\
& \le &  \sum_{k=0}^{+\infty}\P( \exists t\in [k,k+1]: \; \eta_{k}\subset B_{M'k+s/2} \text{ et } \eta_t\not\subset B_{Mt+s}). \label{olili}
\end{eqnarray}
Mais alors, si $C'>0$ est une constante telle que $\Card{B_t}\le C'(1+t)^d$ et si $A,B$ sont les constantes du lemme~\ref{Richardson1},
\begin{eqnarray}
&& \P( \exists t\in [k,k+1]: \; \eta_{k}\subset B_{M'k+s/2} \text{ et } \eta_t\not\subset B_{Mt+s})\nonumber \\
& \le &\P(\eta_{k}\subset B_{M'k+s/2} \text{ et } \eta_{k+1}\not\subset B_{Mk+s})\nonumber\\
&   {\le} & \Card{B_{M'k+s/2}} \P(\eta_1\not\subset B_{ k(M-M')+s/2}) \label{rouge}\\
& \le & C'(1+ M'k+s/2)^dA\exp(-B(k(M-M')+s/2))\nonumber\\
& \le & AC' (1+s/2)^d\exp(-Bs/2)(1+M'k)^d\exp(-B(k(M-M'))\nonumber.
\end{eqnarray}
L'inégalité $(\ref{rouge})$ vient de la propriété de Markov et de l'additivité du processus de contact.
Comme la série de terme général $(1+M'k)^d\exp(-B(k(M-M'))$ converge, le résultat souhaité découle de~(\ref{olala}) et~(\ref{olili}).
\finpreuve \end{proof}

\subsection{Un procédé de redémarrage}

On va utiliser ici un argument dit de redémarrage, que l'on peut  résumer comme suit. 
On couple le système que l'on souhaite étudier (système fort) avec un système qu'il domine stochastiquement (système faible)
et que l'on connaît mieux. On peut alors transporter un certain nombre de propriétés
du système connu à celui que l'on doit étudier: on laisse évoluer les deux systèmes conjointement, et à chaque
fois que le plus faible  meurt et que le plus fort est vivant, on fait repartir 
une copie du plus faible toujours couplée au plus fort.
Ainsi, soit les deux processus meurent avant qu'on ait pu trouver un processus faible capable de survivre, et dans ce cas le contrôle des grands temps de survie du faible peut se transposer sur le fort, soit le plus fort survit indéfiniment et on finit par le coupler avec un faible qui survit. Dans ce cas, un contrôle du temps nécessaire pour un redémarrage réussi permet de transférer des
propriétés du processus faible lorsqu'il survit sur le processus fort.

Cette technique est déjà ancienne; on la trouve par exemple chez Durrett~\cite{MR757768}, section 12, sous une forme très pure. Elle est aussi utilisée par Durrett et Griffeath~\cite{MR656515}, afin de transporter des contrôles connus pour le processus de contact en dimension $1$ au processus de contact en dimension supérieure. Nous allons ici l'utiliser en couplant le processus de contact en environnement inhomogène $\lambda \in \Lambda$ avec le processus de contact avec taux de naissance constant $\lambda_{\min}$. C'est ici que l'hypothèse $\lambda_{\min}>\lambda_c(\Zd)$ est importante. 

Pour ce faire, nous allons coupler des familles de processus ponctuels de Poisson. Fixons  $\lambda \in \Lambda$. On peut construire une mesure de probabilité $\tilde \P_{\lambda}$ sur $\Omega \times \Omega$ sous laquelle
\begin{itemize}
\item La première coordonnée $\omega$ est une famille $((\omega_e)_{e \in \Ed}, (\omega_z)_{z \in \Zd})$ de processus ponctuels de Poisson , d'intensités respectives $(\lambda_e)_{e \in \Ed}$ pour les processus indexés par les arêtes, et d'intensité $1$ pour les processus indexés par les sites.
\item La  seconde coordonnée $\eta$ est une famille $((\eta_e)_{e \in \Ed}, (\eta_z)_{z \in \Zd})$ de processus ponctuels de Poisson , d'intensité $\lambda_{\min}$ pour les processus indexés par les arêtes, et d'intensité $1$ pour les processus indexés par les sites.
\item Les processus ponctuels de Poisson indexés par les sites (les temps de mort) coïncident: pour tout $z \in \Zd$, $\eta_z=\omega_z$.
\item Les processus ponctuels de Poisson indexés par les arêtes (les temps des éventuelles naissances) sont couplés: pour tout $e \in \Ed$, le support de $\eta_e$ est inclus dans celui de $\omega_e$.
\end{itemize}
On note $\xi^A=\xi^A(\omega,\eta)$ le processus de contact dans l'environnement 
$\lambda$ partant de $A$ construit avec la famille de processus de Poisson $\omega$, et
$\zeta^B=\zeta^B(\omega,\eta)$ le processus de contact dans l'environnement 
$\lambda_{\min}$ partant de $B$ construit avec la famille de processus de Poisson $\eta$.
Si $B \subset A$, alors on a $\tilde \P_{\lambda}$ presque sûrement 
$\zeta_t^B \subset \xi_t^A$ pour tout $t \ge 0$. On peut remarquer que le processus $(\xi^A,\zeta^B)$ est un processus de Markov.

On introduit les temps de vie de ces deux processus:
$$
\tau  =  \inf\{t\ge 0: \; \xi^0_t=\varnothing\} \text{ et, pour }x\in\Zd, \; 
\tau'_x  =  \inf\{t\ge 0: \; \zeta^x_t=\varnothing\}. 
$$
Remarquons que la loi de $\tau'_x$ sous $\tilde{\P}_{\lambda}$ est la loi de $\tau_x$ sous $\P_{\lambda_{\min}}$; cette loi est en fait indépendante du point de départ du processus, puisque le modèle à taux de naissance constant est invariant par translation.

On définit par récurrence une suite de temps d'arrêt $(u_k)_{k \ge 0}$ et une suite de points $(z_k)_{k \ge 0}$ en posant $u_0=0$, $z_0=0$, et pour tout $k \ge 0$:
\begin{itemize}
\item si $u_k<+\infty$ et $\xi_{u_k}\neq \varnothing$, alors $u_{k+1}=\tau'_{z_k} \circ \theta_{u_k}$;
\item si $u_k=+\infty$ ou si $\xi_{u_k}= \varnothing$, alors $u_{k+1}=+\infty$;
\item si $u_{k+1}<+\infty$ et $\xi_{u_{k+1}}\neq \varnothing$, alors $z_{k+1}$ est le plus petit point pour l'ordre lexicographique de $\xi_{u_{k+1}}$;
\item si $u_{k+1}=+\infty$ ou si $\xi_{u_{k+1}}= \varnothing$, alors $z_{k+1}=+\infty$.
\end{itemize}
 Autrement dit, tant que $u_k<+\infty$ et $\xi_{u_k}\neq \varnothing$, on prend dans 
$\xi_{u_k}$ le point $z_k$ le plus petit pour  l'ordre lexicographique, et on regarde le temps de vie du processus le plus faible, c'est à dire $\zeta$, partant de $z_k$ au temps $u_k$. Le procédé de redémarrage peut s'arrêter pour deux raisons:
soit on trouve un $k$ tel que $u_k<+\infty$ et $\xi_{u_k}= \varnothing$, ce qui implique que le processus le plus fort (qui contient le faible) meurt exactement au temps $u_k$; 
soit on trouve un $k$ tel que $u_k<+\infty$, $\xi_{u_k} \neq \varnothing$, et $u_{k+1}=+\infty$. Dans ce deuxième cas, on a trouvé un point $z_k$ tel que le processus faible, partant de $z_k$ au temps $u_k$, survit, ce qui implique en particulier que le fort, qui le contient, survit.
On pose alors 
$$K=\inf\{n\ge 0:\;u_{n+1}=+\infty\}.$$
Le nom de la variable $K$ est choisi par analogie avec la section~\ref{sigma}.
Cette section étant indépendante du reste de l'article, la confusion ne devrait cependant pas être possible. Il ressort de la discussion précédente que
\begin{equation}
\label{arev}
(\tau=+\infty \Longleftrightarrow \xi^0_{u_K}\neq \varnothing) \quad \text{ et si } \tau<+\infty, \text{ alors } u_K=\tau.
\end{equation}

On regroupe dans le lemme suivant les estimées sur le procédé de redémarrage nécessaires pour démontrer la proposition~\ref{propuniforme}. On rappelle que $\rho$ est introduit dans l'équation~(\ref{uniftau}).
\begin{lemme}
\label{restart}
On se place dans le cadre précédent. Alors
\begin{itemize}
\item $\forall \lambda \in \Lambda\quad \forall n\in\N\quad\tilde{\P}_\lambda(K>n)\le(1-\rho)^n$.
\item   $\forall B \in \mathcal B(\mathcal D)\quad\tilde{\P}_\lambda(\tau=+\infty, \zeta^{z_K}\circ \theta_{u_K} \in B) = \P_\lambda(\tau=+\infty)\Pbarre_{\lambda_{\min}}(\xi^0 \in B)$.
\item Il existe $\alpha,\beta>0$ tel que pour tout $\lambda \in \Lambda$, $\tilde{\E}_{\lambda}(\exp(\alpha u_K))<\beta$.
\end{itemize}
\end{lemme}

\begin{proof}
D'après la propriété de Markov forte, on a
\begin{eqnarray*}
\tilde{\P}_\lambda(K\ge n+1) & = & \tilde{\P}_\lambda(u_{n+1}<+\infty) \\
& = & \tilde{\P}_\lambda(u_n<+\infty, \; \xi_{u_n} \neq \varnothing, \; \tau'_{z_n} \circ\theta_{u_n})<+\infty) \\
& \le & \tilde{\P}_{\lambda}(u_n<+\infty) (1-\rho)=\tilde{\P}_\lambda(K \ge n)(1-\rho).
\end{eqnarray*}
Ainsi, $K$ a une queue sous-géométrique, ce qui montre le premier point. En particulier, $K$ est presque sûrement fini.

En utilisant~(\ref{arev}) et la propriété de Markov forte, on a encore
\begin{eqnarray*}
&& \tilde{\P}_\lambda(\tau=+\infty, \zeta^{z_K}\circ {\theta}_{u_K} \in B) 
 =  \tilde{\P}_\lambda(\xi_{u_K}\neq \varnothing, \zeta^{z_K}\circ {\theta}_{u_K} \in B) \\
& = & \sum_{k=0}^{+\infty}\sum_{z \in \Zd} \tilde{\P}_\lambda(K=k, \; \xi^0_{u_k}\neq \varnothing, \; z_k=z, \; \zeta^{z_K}\circ {\theta}_{u_K} \in B) \\
& = & \sum_{k=0}^{+\infty}\sum_{z \in \Zd} \tilde{\P}_\lambda(u_k<+\infty, \; \xi^0_{u_k}\neq \varnothing, \; z_k=z, \; \tau'_{z_K}  \circ \theta_{u_k}=+\infty, \; \zeta^{z_K}\circ {\theta}_{u_K} \in B) \\
& = & \sum_{k=0}^{+\infty}\sum_{z \in \Zd} \tilde{\P}_\lambda(u_k<+\infty, \; \xi^0_{u_k}\neq \varnothing, \; z_k=z)\P_{\lambda_{\min}}(\tau=+\infty, \; \xi^{0} \in B) \\
& = & \P_{\lambda_{\min}}( \tau=+\infty, \; \xi^{0} \in B) \sum_{k=0}^{+\infty}\tilde{\P}_\lambda(u_k<+\infty, \; \xi^0_{u_K}\neq \varnothing).
\end{eqnarray*}
En prenant pour $B$ l'ensemble des trajectoires, on identifie:
$$\tilde{\P} (\tau=+\infty)= \P_\lambda(\tau=+\infty)=\P_{\lambda_{\min}}(\tau=+\infty)\sum_{k=0}^{+\infty}\tilde{\P}_\lambda(u_k<+\infty,\xi^0_{u_K}\neq \varnothing),$$ 
ce qui nous donne le deuxième point

Comme $\lambda_{\min}>\lambda_c(\Zd)$, les résultats de Durrett et Griffeath~\cite{MR656515} pour les grands~$\lambda$, étendus à tout le régime surcritique par Bezuidenhout et Grimmett~\cite{MR1071804}, assurent l'existence de $A,B>0$ telles que
$$\forall t\ge 0\quad \P_{\lambda_{\min}}(t \le \tau <+\infty) \le A \exp(-Bt),$$
ce qui donne l'existence de moments exponentiels pour $\tau \1_{\{\tau<+\infty\}}$.
Comme $\P_{\lambda_{\min}}(\tau=+\infty)>0$, on peut choisir (en utilisant par exemple le théorème de convergence dominée) un $\alpha>0$ tel que $\E_{\lambda_{\min}}(\exp(\alpha\tau)\1_{\{\tau<+\infty\}})=r<1$. \\
Pour $k\ge 0$, on note
$$S_k=\exp \left( \alpha \sum_{i=0}^{k-1}\tau'_{z_i} \circ \theta_{u_i} \right)\1_{\{u_{k}<+\infty\}}.$$
On remarque que $S_k$ est $\mathcal{F}_{u_{k}}$-mesurable. Soit $k\ge 0$. On a
$$\exp(\alpha u_K)\1_{\{K= k\}}\le S_{k}.$$
Ainsi, en appliquant la propriété de Markov forte au temps $u_{k-1}<+\infty$, on obtient, pour $k \ge 1$
\begin{eqnarray*}
\tilde{\E}_\lambda[\exp(\alpha u_K)\1_{\{ K= k\}}] & \le & \tilde{\E}_\lambda(S_k)=\tilde{\E}_\lambda(S_{k-1}){\E}_{\lambda_{\min}}(\exp(\alpha\tau)\1_{\{\tau<+\infty\}}) \\
& \le & r\tilde{\E}_\lambda(S_{k-1}).
\end{eqnarray*}
Comme $r<1$, on en déduit que $\tilde{\E}_\lambda[\exp(\alpha u_K)]\le \frac{r}{1-r}<+\infty$.
\finpreuve \end{proof}

\subsection{Preuve de la proposition~\ref{propuniforme}}

Les estimées (\ref{richard}) et (\ref{uniftau}) découlent d'une simple comparaison stochastique:

\begin{proof}[\debutpreuve de (\ref{uniftau})]  Il suffit de remarquer que pour tout environnement $\lambda \in \Lambda$ et tout $z\in\Zd$, on a
$$\P_{\lambda}(\tau^z=+\infty) \ge \P_{\lambda_{\min}}(\tau^z=+\infty)=\P_{\lambda_{\min}}(\tau^0=+\infty)>0.$$
\finpreuve \end{proof}

\begin{proof}[\debutpreuve de~(\ref{richard})] On utilise la domination  stochastique du  processus de contact en environnement $\lambda$ par
le modèle de Richardson de paramètre  $\lambda_{\max}$. Pour ce modèle, la croissance au plus linéaire est assurée par~(\ref{voirkesten}). 
\finpreuve \end{proof}

Il nous reste donc à démontrer (\ref{grosamasfinis}), (\ref{retouche}) et (\ref{petitsouscouple}) par le procédé de redémarrage.
 
\begin{proof}[\debutpreuve de~(\ref{grosamasfinis})]
Soit $\alpha,\beta>0$ données dans le troisième point du lemme~\ref{restart}. Rappelons que, sur $\{\tau<+\infty\}$, $u_{K}=\tau$. On a,
pour tout $\lambda \in \Lambda$, pour tout $t>0$,
\begin{eqnarray*}
\P_\lambda(t<\tau<+\infty)& =& \P_\lambda(e^{\alpha t}<e^{\alpha \tau},\tau<+\infty)
 =  \tilde{\P}_\lambda(e^{\alpha t}<e^{\alpha u_{K}},\tau<+\infty)\\
& \le& \tilde{\P}_\lambda(e^{\alpha t}<e^{\alpha u_{K} })\le e^{-\alpha t} \tilde{\E}_\lambda e^{\alpha u_{K}} \le \beta e^{-\alpha t},
\end{eqnarray*}
ce qui termine la preuve.
\finpreuve \end{proof}

\begin{proof}[\debutpreuve de~(\ref{retouche})] Comme $\lambda_{\min}>\lambda_c(\Zd)$, les résultats de Durrett et Griffeath~\cite{MR656515} pour les grands $\lambda$, étendus à tout le régime surcritique par Bezuidenhout et Grimmett~\cite{MR1071804}, assurent l'existence de constantes $A,B,c>0$ telles que, pour tout $y \in \Zd$, pour tout $t \ge 0$,
\begin{equation}
\label{blibli}
\Pbarre_{\lambda_{\min}}\left( t(y) \ge \frac{\|y\|}{c}+t \right) \le A \exp(-Bt).
\end{equation}
D'autre part, le contrôle par le modèle de Richardson du lemme~\ref{Richardson2} avec $\lambda_{\max}$ assure l'existence de $A,B,M$ tels que pour tout $\lambda \in \Lambda$, pour tout $s \ge 0$,
\begin{equation}
\label{blabla}
\P_{\lambda}(\exists t\ge0, \; \xi^0_t \not\subset B_{Mt+s}) \le A\exp(-Bs).
\end{equation}
Quitte à diminuer $c$ ou à augmenter $M$, on peut en outre supposer que $\frac{c}{M}\le 1$. Maintenant,
\begin{eqnarray*}
&& \tilde{\P}_{\lambda} \left(t(y) \ge \frac{\|y\|}c +t, \; \tau=+\infty\right) \\
& \le & \tilde{\P}_{\lambda} \left( u_K \ge \frac{tc}{6M} \right)+ \tilde{\P}_{\lambda} \left( u_K \le \frac{tc}{6M}, \; \xi^0_{u_K} \not\subset B_{tc/3} \right) \\
&& +\tilde{\P}_{\lambda} \left(\tau=+\infty, \; u_K \le \frac{tc}{6M}, \; \xi^0_{u_K} \subset B_{tc/3}, \; t(y) \ge \frac{\|y\|}c +t\right).
\end{eqnarray*}
Le premier terme est bien contrôlé par l'existence de moments exponentiels pour $u_K$ donnée par le troisième point du lemme~\ref{restart} de redémarrage: il existe $C,\alpha>0$ tels que pour tout $\lambda \in \Lambda$, pour tout $t\ge 0$,
$$\tilde{\P}_{\lambda} \left(u_K \ge \frac{tc}{6M}\right) \le C\exp\left(-\frac{\alpha c t}{6M}\right).$$
Le second terme est contrôlé à l'aide de~(\ref{blabla}):
\begin{eqnarray*}
\tilde{\P}_{\lambda} \left(u_K \le \frac{tc}{6M}, \; \xi^0_{u_K} \not\subset B_{tc/3} \right) & \le & \P_{\lambda}(\exists t\ge0, \; \xi^0_t \not\subset B_{Mt+\frac{tc}6}) \le A\exp\left(-B\frac{tc}6\right).
\end{eqnarray*}
Il reste à contrôler le dernier terme. On note ici
$$t'(y)=\inf\{t \ge 0:\; y \in \zeta^0_t\}.$$
Rappelons que si $\tau=+\infty$, alors $\xi_{u_K} \neq \varnothing$ et $z_K$ est bien défini.
Comme $t(y)$ est un temps d'entrée et que $\xi^0_t\supset\zeta^0_t $ pour tout $t$, on a, sur $\{\tau=+\infty\}$,
$$t(y)\le  u_{K}+t'(y-z_K) \circ T_{z_K} \circ \theta_{u_K}.$$
Si $u_K \le \frac{tc}{6M} \le \frac{t}6$, alors 
$t(y) \le \frac{t}{6}+t'(y-z_K)\circ T_{z_K} \circ \theta_{u_K}$.
Si, de plus, $\xi^0_{u_K} \subset B_{tc/3}$, on a $\|y \| \ge \|y-z_K\|-\frac{tc}3$,
ce qui donne, avec deuxième point du lemme~\ref{restart},
\begin{eqnarray*}
&& \tilde{\P}_{\lambda} \left(\tau=+\infty, \; u_K \le \frac{ct}{6M}, \; \xi^0_{u_K} \subset B_{ct/3}, \; t(y) \ge \frac{\|y\|}c +t\right) \\
& \le & \tilde{\P}_{\lambda}\left(\tau=+\infty, 
\;t'(y-z_K) \circ T_{z_K} \circ \theta_{u_K} \ge \frac{\|y-z_K\|}c +\frac{t}{2}\right) \\
& \le & \P_\lambda(\tau=+\infty) \sup_{z \in \Zd}\Pbarre_{\lambda_{\min}}\left(t(y-z)  \ge \frac{\|y-z\|}c +\frac{t}2\right)
 \le  A\exp(-Bt/2),
\end{eqnarray*}
où la dernière égalité provient de~(\ref{blibli}). Ceci termine la preuve.
\finpreuve \end{proof}

\begin{proof}[\debutpreuve de (\ref{petitsouscouple})]
Soit $s \ge 0$, et notons $n$ la partie entière de $s$. Soit $\gamma>0$ fixé, dont la valeur sera précisée ultérieurement.
\begin{eqnarray*}
&& \Pbarre(0\notin K'_s) 
 =  \Pbarre(\exists t \ge s: \; 0 \not\in K_t) \\
& \le & \sum_{k=n}^{+\infty}\Pbarre(B_{\gamma k}\not\subset K_k)+
\sum_{k=n}^{+\infty}\Pbarre(B_{\gamma k}\subset K_k, \; \exists t\in [k,k+1) \text{ tel que } 0\not\in K_t).
\end{eqnarray*}

Commençons par contrôler la deuxième somme. Fixons $k \ge n$. Supposons que  $B_{\gamma k}\subset K_k$ et considérons $t\in [k,k+1) \text{ tel que } 0\not\in K_t$. Il existe donc  $x\in\Zd$ tel que $0\in \xi^x_t\backslash \xi^0_t$. Comme $0\in \xi^x_t$ et $t \ge k$, il existe $y \in \Zd$ tel que $y \in \xi^x_k$ et $0 \in \xi^y_{t-k} \circ \theta_k$. Si $y \in B_{\gamma k}\subset K_k$, alors $\xi^{0}_k(y)=\xi^{\Zd}_k(y)=1$, ce qui implique que $y \in \xi^{0}_k$, ce qui à son tour, vu que $0 \in \xi^y_{t-k} \circ \theta_k$, implique que
$0\in \xi^0_t$, et contredit l'hypothèse $0\not\in \xi^0_t$. Ainsi, nécessairement, $y \not\in B_{\gamma k}$, et donc:
\begin{eqnarray*}
&& \Pbarre_{\lambda}(B_{\gamma k}\subset K_k,\;\exists t\in [k,k+1)\text{ tel que } 0\not\in K_t) \\
& \le & \frac1{\P_\lambda(\tau=+\infty)}\P_{\lambda}\left(\theta_k^{-1} \left( 0\in \miniop{}{\cup}{s\in [0,1]} \xi^{\Zd \backslash B_{\gamma k}}_s \right) \right) \\
& \le &  \frac1\rho\P_{\lambda}\left(0\in \miniop{}{\cup}{s\in [0,1]} \xi^{\Zd \backslash B_{\gamma k}}_s\right) \\
& \le & \frac1\rho \P_{\lambda}\left(\miniop{}{\cup}{s\in [0,1]} \xi^{0}_s\not\subset B_{\gamma k}\right)=  \frac1{\rho}\P_{\lambda}(H_1^0\not\subset B_{\gamma k}).
\end{eqnarray*}
Comme le modèle de Richardson de paramètre $\lambda_{\max}$ domine stochastiquement le processus de contact en environnement $\lambda$, on contrôle ce dernier terme à l'aide du lemme~\ref{Richardson1}.

Pour contrôler la première somme, il suffit de montrer qu'il existe des constantes strictement positives $A,B,\gamma$ telles que pour tout
$\lambda\in\Lambda$, pour tout $ t\ge0$
\begin{equation}
\P_{\lambda}(B_{\gamma t}\not\subset K_t, \; \tau^0=+\infty) \le A\exp(-B t).
\label{souscouple}
\end{equation}
La constante $\gamma$ dont nous avions reporté la définition est ainsi déterminée.

Le nombre de points contenus dans une boule étant polynomial en le rayon de la boule, il
suffit de montrer qu'il existe des constantes $A,B,c'>0$ telles que pour tout $t\ge 0$, pour tout $x \in \Zd$,
\begin{equation}
\label{monbutla}
\|x\|\le c't \Longrightarrow \tilde{\P}_{\lambda}(\xi^{0}(t)\ne\varnothing,x\in \xi^{\Zd}(t)\backslash \xi^{0}(t))\le A\exp(-Bt).
\end{equation}
Pour démontrer~(\ref{monbutla}), on va s'appuyer sur le résultat suivant, obtenu par Durrett~\cite{MR1117232} comme conséquence de la construction de Bezuidenhout et Grimmett~\cite{MR1071804}: si $\xi^0$ et $\tilde{\xi}^x$ sont deux processus de contact indépendants
de paramètre $\lambda>\lambda_c(\Zd)$, partant respectivement de $0$ et de $x$, alors il existe des constantes $A,B,\alpha$ strictement positives telles que pour tout $t\ge 0$, pour tout $x \in \Zd$,
\begin{equation}
\label{durrett}
\|x\|\le\alpha t \Longrightarrow\quad\P(\xi^0_t\cap\tilde{\xi}^x_t=\varnothing, \tilde{\xi}^x_t\ne \varnothing,{\xi}^0_t\ne \varnothing)\le A\exp(-Bt).
\end{equation}
Soient $\alpha$ et $M$ les constantes respectivement données par les équations~(\ref{durrett}) et~(\ref{richard}).
On pose $c'=\alpha/2$ et on choisit $\epsilon>0$ tel que $c'+2\epsilon M\le \alpha$.

Soient $a\in B^0_{\alpha t/4}$ et  $b\in B^x_{\alpha t/4}$. On pose
$$\alpha_{a,s}=\zeta^a_s\circ  \theta_{\epsilon t/2}\text{ et }\beta_{b,s}=\{y\in\Zd: \; b\in \zeta^{y}_s\circ \theta_{t(1-\epsilon/2)-s}\}.$$
Alors $(\alpha_{a,s})_{0\le s\le t/2(1-\epsilon)}$ et $(\beta_{a,s})_{0\le s\le t/2(1-\epsilon)}$
sont deux processus de contact indépendants de taux de naissance $\lambda_{\min}$ constant, partant respectivement de $a$ et $b$. 
Le processus $(\beta_{a,s})_{0\le s\le t/2(1-\epsilon)}$ est un processus de contact, mais pour lequel on a  retourné l'axe temporel dans la construction avec les processus de Poisson.
De manière analogue, on pose
$$\hat{\xi}_s^x=\{y\in\Zd :\; x\in \xi_s^y\circ \theta_{t-s}\}.$$
La loi de $(\hat{\xi}^x_{s})_{0\le s\le t/2}$ coïncide avec la loi de
 $({\xi}^x_{s})_{0\le s\le t/2}$.
On s'appuie sur les remarques  suivantes:
 \begin{itemize}
 \item Si $a\in\xi^0_{\epsilon t/2}$, que $\alpha_{a,(1-\epsilon) t/2}\cap \beta_{b,(1-\epsilon )t/2}\ne\varnothing$
et que
$b\in \hat{\xi}^x_{\epsilon t/2}$, alors $x\in\xi_t^0$.
\item Si $x\in\xi_t^{\Zd}$, alors $\hat{\xi}^x_{t/2}$  est non-vide.
 \item Si ${\xi}^0_{t}$ est non-vide, alors ${\xi}^0_{t/2}$ est non-vide.
 \end{itemize}
Ainsi, en posant
\begin{eqnarray*}
E^0 & = & \left\{\xi^0_{t/2}\ne\varnothing\right\}\backslash 
\left\{
\exists a\in B^0_{\alpha  t/4} \cap \xi^0_{\epsilon t/2}:\;
\alpha_{a,(1-\epsilon) t/2}\ne\varnothing
\right\} \\
\text{et } \quad \hat{E}^x & = & \left\{\hat{\xi}^x_{t/2}\ne\varnothing\right\}\backslash 
\left\{
\exists b\in B^x_{\alpha  t/4} \cap \hat{\xi}^x_{\epsilon t/2}:\;
\beta_{b,(1-\epsilon) t/2}\ne\varnothing
\right\}, 
\end{eqnarray*}
on obtient
\begin{eqnarray}
\label{decouplage}
\tilde{\P}_{\lambda}(\xi^{0}_t\ne\varnothing,\; x\in \xi^{\Zd}_t\backslash \xi^{0}_t) 
 & \le & \tilde{\P}_{\lambda}( \xi^{0}_{t/2} \ne \varnothing, \; \hat{\xi}^x_{t/2} \ne \varnothing,\; \xi^{0}_{t/2}\cap \hat{\xi}^x_{t/2}=\varnothing) \\
 & \le & \tilde{\P}_{\lambda}(E^0)+\tilde{\P}_{\lambda}(\hat{E}^x) +S \nonumber,
 \end{eqnarray}
 où
 $\displaystyle S
  =  \mathop{\sum_{a\in B^0_{\alpha  t/4}}}_{b\in B^x_{\alpha  t/4}}\tilde{\P}_{\lambda}\left(\alpha_{a,\frac{(1-\epsilon) t}2}\ne\varnothing, \; \beta_{b,\frac{(1-\epsilon) t}2}\ne\varnothing,\; \alpha_{a,\frac{(1-\epsilon) t}2}\cap\beta_{b,\frac{(1-\epsilon) t}2}=\varnothing\right)$. \\
Pour chaque couple $(a,b)$ apparaissant dans $S$, on a
$\|a-b\|\le \|a\|+\|b-x\|+\|x\|\le \alpha t/4+ \alpha t/4+\alpha t/2=\alpha t$, ce qui permet d'utiliser~(\ref{durrett}), et donne l'existence de constantes $A,B,C'>0$ telles que 
$$S \le C'(1+\alpha t/4)^{2d}A\exp(-B(1-\epsilon)t/2).$$
En retournant à nouveau le temps, on voit
que $\tilde{\P}_{\lambda}(\hat{E}^x)=\tilde{\P}_{\T{x}\lambda}(E^0)$; il suffit donc de contrôler
$\tilde{\P}_{\lambda}(E^0)$ uniformément en $\lambda$.
Posons
$$E_1=\{\xi^{0}_{t/2}\ne\varnothing\}\backslash \{\exists a\in\Zd: \; a\in\xi^0_{\epsilon t/2},\; \alpha_{a,(1-\epsilon) t/2}\ne\varnothing\}.$$
On a 
$\tilde{\P}_{\lambda}(E^0)\le \tilde{\P}_{\lambda}(E_1)+\tilde{\P}_{\lambda}(\xi^0_{\epsilon t/2}\not\subset B^0_{\alpha t/4})$.
D'après le choix de $\epsilon$ et l'inégalité~(\ref{richard}), on a
$$\forall \lambda\in\Lambda\quad\forall t\ge 0\quad \tilde{\P}_{\lambda}(\xi^0_{\epsilon t/2}\not\subset B(0,\alpha t/4))\le A\exp(-B\epsilon t/2).$$
\`A l'aide du lemme~\ref{restart} de redémarrage, on voit que
$$\tilde{\P}_{\lambda}(u_K> \epsilon t/2)\le \beta\exp(-\alpha \epsilon t/2).$$
Supposons donc que $u_K\le\epsilon t/2$ et que $\xi^{0}_{t/2}\ne\varnothing$: $z_K$ est donc bien défini et 
on a $\tau'_{z_K} \circ \theta_{u_K}=+\infty$. 
Ainsi, il existe une branche d'infection infinie dans le processus couplé
en environnement $\lambda_{\min}$ partant de $\xi_{u_K}^0$. Cette branche
contient au moins un point $a\in\xi^0_{(1-\epsilon)t/2}$.
Par construction $a\in\xi^0_{(1-\epsilon)t/2}$ et
$\alpha_{a,(1-\epsilon) t/2}\ne\varnothing$, ce qui achève la preuve de~(\ref{souscouple}).
\finpreuve \end{proof}

\noindent
\textbf{Remarque} Au cours de cette preuve, on a montré que pour tout $\lambda \in \Lambda$,
$$\lim_{t \to+\infty} \tilde{\P}_{\lambda}( \xi^{0}_{t} \ne \varnothing, \; \hat{\xi}^x_{t} \ne \varnothing,\; \xi^{0}_{t}\cap \hat{\xi}^x_{t}=\varnothing)=0,$$
ce qui est l'ingrédient essentiel de la preuve du théorème~\ref{thCC} de convergence complète. On pourra se reporter à l'article de Durrett~\cite{MR1117232} pour les détails de la preuve dans le cas du processus de contact classique.

\section*{Appendice: preuve des théorèmes ergodiques presque sous-additifs~\ref{therg} et~\ref{thergdeux}}

\begin{proof}[\debutpreuve du théorème~\ref{therg}]
Posons $a_p=C_p^{1/\alpha}$ et $u_n=\E[f_n]$: on a pour tous $n,p \in \N$
$\E[r_{n,p}^+]\le (\E[(r_{n,p}^+)^{\alpha}])^{1/\alpha}\le C_p^{1/\alpha}=a_p,$
d'où 
$$u_{n+p}\le u_n+u_p+\E[g_p]-\E[r_{n,p}]\le u_n+u_p+\E [g_p]+a_{p}.$$
Le terme général d'une série convergente tend vers $0$, donc $C_p=o(p^{\alpha})$, soit $a_p=o(p)$. Comme $\frac{a_n+\E [g_n]}n$ tend vers $0$, la convergence de $u_n/n$ est classique (voir par exemple Derriennic~\cite{MR704553}). La limite $\mu$ est finie car $u_n\ge cn$ pour tout $n$.

On va montrer que $\underline f=\miniop{}{\liminf}{n \to +\infty}\frac{f_{n}}n $ domine stochastiquement une variable aléatoire dont
l'espérance dépasse $\mu$.

Pour toute variable aléatoire $X$, on note $\mathcal L(X)$ sa loi sous $\P$. 
On note $\mathcal{K}$ l'ensemble des mesures de probabilités sur $\R_+^{\N^*}$ dont toutes les lois marginales~$m$ vérifient:
$$\forall t>0\quad m(]t,+\infty[)\le \P(f_1+g_1>t/2)+C_1(2/t)^{\alpha}.$$
On définit, pour tout $k \ge 1$,
$$\Delta_k=f_{k+1}-f_{k},$$
et on note $\Delta$ le processus $\Delta=(\Delta_k)_{k \ge 1}$.
Pour tout $k\in \N$, la sous-additivit\'e assure que $\Delta_k\le (f_1+g_1)\circ\theta_k+r_{k,1}$, 
donc, pour tout $t>0$,
\begin{eqnarray*}
\P(\Delta_k>t) & \le & \P((f_1+g_1)\circ\theta_k>t/2)+\P(r_{k,1}^+>t/2) \\
& \le & \P(f_1+g_1>t/2)+C_1(2/t)^{\alpha}.
\end{eqnarray*}
Ceci assure que $\Delta \in \mathcal K$. 

En notant $s$ le shift: $s((u_k)_{k \ge 0})=(u_k)_{k \ge 1}$,
on regarde alors la suite de probabilités sur $\R^{\N^*}$:
$$(L_n)_{n \ge 1}=\left( \frac1n \sum_{j=1}^{n} \mathcal L(s^j \circ \Delta) \right)_{n \ge 1}.$$
Comme $\mathcal{K}$ est convexe et invariant par $s$, la suite $(L_n)_{n\ge 1}$ prend ses valeurs dans $\mathcal{K}$. 
Soient $n,k\ge 1$.
\begin{eqnarray*}
\int \pi_k(x)\ dL_{n}(x)
   & = &  \frac1{n} \sum_{j=1}^{n} \E (\pi_k(s^j \circ \Delta)) \\
& = &  \frac1{n} \sum_{j=1}^{n} \E (f_{k+j+1}-f_{k+j}) 
 =  \frac1{n} (\E [f_{n+k+1}]-\E[f_{k+1}]).
\end{eqnarray*}
Posons $M_k=\miniop{}{\sup}{n\ge 1} \frac1{n} |\E [f_{n+k+1}]-\E[f_{k+1}]|$.
La convergence de $u_n/n$ entraîne que $M_k$ est fini. De m\^eme, par presque sous-additivit\'e,
\begin{eqnarray*}
\int \pi_k^+(x)\ dL_{n}(x)
   & = &  \frac1{n} \sum_{j=1}^{n} \E (\pi_k^+(s^j \circ \Delta)) \\
& = &  \frac1{n} \sum_{j=1}^{n} \E [(f_{k+j+1}-f_{k+j})^+] 
 \le \E [f_1^+]+\E[g_1^+]+a_1.
\end{eqnarray*}
Ainsi, on a 
\begin{eqnarray*}\int |\pi_k(x)|\ dL_{n}(x)&\le&  \int 2\pi_k^+(x)\ dL_{n}(x)+ |\int \pi_k(x)\ dL_{n}(x)|\\ &\le & M_k+2\E [f_1^+]+2\E[g_1^+]+2a_1.
\end{eqnarray*}
Soit $\mathcal{K}'$ l'ensemble des lois $m$ sur $\R^{\N^*}$ tel que pour tout $k$,
$\int |\pi_k| \ dm\le 2 M_k+\E [f_1^+]+\E[g_1^+]+a_1$.
L'ensemble $\mathcal{K}'$ est compact pour la topologie de la convergence en loi et
la suite $(L_n)_{n\ge 1}$ prend ses valeurs dans $\mathcal{K}'$. Soit donc $\gamma$ un point d'accumulation de $(L_n)_{n \ge 1}$ et $(n_k)_{k\ge 1}$ une suite d'indices telle que $L_{n_k}\Longrightarrow\gamma$.
Par construction, $\gamma$ est invariante par le shift $s$.

Maintenant, la suite des lois de la première coordonnée $\pi_1(x)$ sous $(L_{n_k})_{k \ge 0}$ converge faiblement vers la  loi de la première coordonnée  sous $\gamma$. Par ailleurs, par définition de $\mathcal{K}$, cette famille de lois a une partie positive uniformément intégrable, donc
$\int \pi_1^+\ d\gamma=\lim \int \pi_1^+ \ dL_{n_k}$.
Cependant, le lemme de Fatou nous dit que
$\int \pi_1^-\ d\gamma\le\miniop{}{\liminf}{k\to +\infty} \E \pi_1^- \ dL_{n_k}$, d'où finalement
$$\int \pi_1\ d\gamma\ge\miniop{}{\liminf}{k\to +\infty}\int \pi_1 \ dL_{n_k}=\mu.$$

Soit $Y=(Y_k)_{k \ge 1}$ un processus de loi $\gamma$.
Comme $\gamma$ est invariante par le shift~$s$, le théorème de Birkhoff nous dit
que la suite $(\frac1n \sum_{k=1}^{n}Y_k)_{n \ge 1}$ converge p.s.\  vers une variable aléatoire réelle $Y_{\infty}$, qui vérifie donc $\E(Y_{\infty})=\int \pi_1\ d\gamma\ge\mu$.

Il nous reste à voir que la loi de $Y_{\infty}$ est stochastiquement dominée par la loi de $ \underline{f}=\miniop{}{\liminf}{n \to +\infty} \frac1{n}f_{n}$.
On va montrer que pour tout $a\in\R$, $\P(Y_{\infty}>a)\le \P(\underline{f}>a)$.
Par continuité à droite, il suffit d'obtenir l'inégalité pour $a$ dans une partie dense de $\R$. On supposera donc que $a$ est un point de continuité de la loi de $ \underline{f}$.

$$\left\{Y_{\infty}>a\right\}=\left\{\miniop{}{\liminf}{n\to +\infty}\frac{Y_1+\dots+Y_n}{n}>a\right\}=\miniop{}{\bigcup}{k\ge 1}\left\{ \miniop{}{\inf}{n\ge k}\frac{Y_1+\dots+Y_n}{n}>a\right\}.$$
D'où 
\begin{eqnarray*}
\P(Y_{\infty}>a) & = & \miniop{}{\limsup}{k\to +\infty}\P_{Y}\left(\miniop{}{\inf}{i\ge k} \frac{\pi_1+\dots+\pi_i}{i}>a\right)\\
& = & \miniop{}{\limsup}{k\to +\infty}\miniop{}{\inf}{n\ge k}\P_{Y}\left(\miniop{}{\inf}{k\le i\le n} \frac{\pi_1+\dots+\pi_i}{i}>a\right)\\
& \le & \miniop{}{\limsup}{k\to +\infty}\miniop{}{\inf}{n \ge k} \miniop{}{\liminf}{K \to +\infty}\frac{1}{n_K}\sum_{j=1}^{n_K}\P
\left(\miniop{}{\inf}{k\le i\le n} \frac{\pi_1+\dots+\pi_i}{i} \circ s^j\circ\Delta >a\right).
\end{eqnarray*}
Soit $\epsilon>0$. On a, pour $k,n,j$ fix\'es,
\begin{eqnarray*} & &\P\left(\miniop{}{\inf}{k\le i\le n} \frac{\pi_1+\dots+\pi_i}{i} \circ s^j\circ\Delta >a\right)\\
& = & \P\left(\miniop{}{\inf}{k\le i\le n} \frac{f_{j+i+1}-f_{j+1}}{i}  >a\right)\\
& \le & \P\left(\miniop{}{\inf}{k\le i\le n} \frac{(f_{i}+g_i)\circ\theta_{j+1}+r_{j+1,i}}{i} >a\right)\\
& \le & \P\left(\miniop{}{\inf}{k\le i\le n} \frac{(f_{i}+g_i)\circ\theta_{j+1}}{i} >a-\epsilon\right)+\P\left(\miniop{}{\sup}{i\ge k} \frac{r_{j+1,i}}{i} >\epsilon\right)\\
\end{eqnarray*}
D'une part, on a
\begin{eqnarray*}
\P\left(\miniop{}{\sup}{i\ge k} \frac{r_{j+1,i}}{i} >\epsilon\right) & \le &
\P\left(\miniop{}{\sum}{i\ge k} \left(\frac{r^+_{j+1,i}}{i}\right)^{\alpha} >\epsilon^{\alpha}\right)
 \le \epsilon^{-\alpha}\miniop{}{\sum}{i\ge k}\frac1{i^{\alpha}} \E[(r^+_{j+1,i})^{\alpha}] \\
 &\le & \epsilon^{-\alpha}\miniop{}{\sum}{i\ge k} \frac{C_i}{i^{\alpha}}.
\end{eqnarray*}
Notons que ce terme est ind\'ependant de $j$ et de $n$.
D'autre part,
\begin{eqnarray*} \P\left(\miniop{}{\inf}{k\le i\le n} \frac{(f_{i}+g_i)\circ\theta_{j+1}}{i} >a-\epsilon\right) & = &  \P\left(\miniop{}{\inf}{k\le i\le n} \frac{f_{i}+g_i}{i} >a-\epsilon\right),
\end{eqnarray*}
qui est ind\'ependant de $j$. Ainsi quel que soit $\epsilon>0$, on a pour tous $n,k$, avec $n\ge k$:
$$ \miniop{}{\liminf}{K \to +\infty}\frac{1}{n_K}\sum_{j=1}^{n_K}\P
\left(\miniop{}{\inf}{k\le i\le n} \frac{\pi_1+\dots+\pi_i}{i} \circ s^j\circ\Delta >a\right)\le \epsilon^{-\alpha}\miniop{}{\sum}{i\ge k} \frac{C_i}{i^{\alpha}}+\P\left(\miniop{}{\inf}{k\le i\le n} \frac{f_{i}+g_i}{i} >a-\epsilon\right),$$puis
\begin{eqnarray*}
\miniop{}{\inf}{n\ge k} \miniop{}{\liminf}{K \to +\infty}\frac{1}{n_K}\sum_{j=1}^{n_K}\P
\left(\miniop{}{\inf}{k\le i\le n} \frac{\pi_1+\dots+\pi_i}{i} \circ s^j\circ\Delta >a\right)&\le& \epsilon^{-\alpha}\miniop{}{\sum}{i\ge k} \frac{C_i}{i^{\alpha}}\\& &+\miniop{}{\inf}{n\ge k}\P\left(\miniop{}{\inf}{k\le i\le n} \frac{f_{i}+g_i}{i} >a-\epsilon\right).
\end{eqnarray*}
Enfin
\begin{eqnarray*}
\P(Y_{\infty}>a)
& \le & \miniop{}{\limsup}{k\to +\infty}\miniop{}{\inf}{n \ge k} \P\left(\miniop{}{\inf}{k\le i\le n} \frac{f_{i}+g_i}{i} >a-\epsilon\right)+\miniop{}{\limsup}{k\to +\infty}\epsilon^{-\alpha}\miniop{}{\sum}{i\ge k} \frac{C_i}{i^{\alpha}} \\
& \le & \miniop{}{\limsup}{k\to +\infty} \P\left(\miniop{}{\inf}{i\ge k} \frac{f_{i}+g_i}{i} >a-\epsilon\right) \\
& \le& \P\left(\miniop{}{\liminf}{i\to +\infty} \frac{f_{i}+g_i}{i} >a-\epsilon\right)
=  \P\left(\miniop{}{\liminf}{i\to +\infty} \frac{f_{i}}{i} >a-\epsilon\right),
\end{eqnarray*}
\'etant donn\'e que $g_i/i$ tend presque sûrement vers $0$.
Faisant alors tendre $\epsilon$ vers zéro, on obtient
$$\P(Y_{\infty}>a)\le \P \left(\miniop{}{\liminf}{i\to +\infty} \frac{f_{i}}{i}\ge a \right)=\P(\underline{f}>a).$$

Reste \`a voir l'invariance par les $\theta_n$ de $\underline f$.
Soit $n\ge 1$ fixé. On a
$$\E \left[ \sum_{p=1}^{+\infty} \left(\frac{r^+_{n,p}}p\right)^{\alpha} \right] 
= \sum_{p=1}^{+\infty} \E \left[\left(\frac{r^+_{n,p}}p \right)^{\alpha} \right] \le\sum_{p=1}^{+\infty}\frac{C_p}{p^{\alpha}}<+\infty.$$ 
En particulier, $\frac{r^+_{n,p}}p$ tend presque sûrement vers $0$ lorsque $p$ tend vers l'infini.
Comme $f_{n+p}\le f_n+f_p\circ\theta_n+g_p\circ\theta_n+r_{n,p}^+$, en divisant par $n+p$ et en faisant tendre $p$ vers $+\infty$, il vient que
$$\underline{f} \le \underline{f} \circ \theta_n\quad\text{p.s.}$$
Comme $\theta_n$ laisse $\P$ invariante, on en déduit classiquement que $\underline{f}$ est invariante par $\theta_n$. 
\finpreuve
\end{proof}

\textbf{Remarques} Dans le présent article, il n'est pas fait usage de la possibilité de prendre $g_p$ non nul. Dans le cas où les $(g_p)$ ne sont pas nuls mais où les
$r_{n,p}$ sont, en revanche, nuls, on obtient un résultat qui ressemble beaucoup au théorème 3 de Schürger~\cite{MR1127716}. 
Comme Schürger~\cite{MR833959}, nous reprenons le procédé de couplage avec un processus ``stationnarisé'' initié par Durrett~\cite{MR586774} et popularisé par Liggett~\cite{MR806224}. Ici, le raisonnement est raffiné en établissant directectement une comparaison stochastique de la variable aléatoire $Y$, et non du processus $(Y_n)_{n\ge 1}$ dont elle est la limite inférieure.

Dans la plupart des théorèmes ergodiques  presque sous-additifs existants, la convergence presque sûre  requiert des conditions assez fortes (du type stationnarité en loi) du défaut de sous-additivité. 
Ici, en revanche, on obtient un comportement presque sûr en considérant une 
condition sur les moments (d'ordre supérieur à 1) du défaut de sous-additivité. A contrario, on sait qu'un contrôle du moment d'ordre 1 du défaut de sous-additivité n'est pas suffisant pour obtenir un comportement presque sûr (voir la remarque de Derriennic~\cite{MR704553} et le contre-exemple de Derriennic et Hachem~\cite{MR939537}).

\begin{proof}[\debutpreuve du théorème~\ref{thergdeux}]
Il reste à montrer que $\displaystyle \E \left(\miniop{}{\limsup}{n \to +\infty} \frac{f_{n}}n \right) \le \mu$. \\
On  fixe $k \ge 1$. En utilisant la sous-additivité, on a pour tout $n\ge 0$ et pour tout $0 \le r \le k-1$:
\begin{eqnarray*}
f_{nk+r}& \le & f_{nk}+(f_r+g_r)\circ\theta_{nk}+r_{nk,r}^+\\
& \le &\left(\sum_{i=0}^{n-1}f_k\circ(\theta_k)^i\right)+\left(f_{nk}-\sum_{i=0}^{n-1}f_k\circ(\theta_k)^i\right)^++(f_r+g_r)\circ\theta_{nk}+r_{nk,r}^+
\end{eqnarray*}
Comme $\theta_k$ laisse $\P$ invariante, le théorème de Birkhoff nous donne la convergence presque sûre et dans $L^1$:
$$\lim_{n \to +\infty} \frac1n \sum_{j=0}^{n-1} \frac{f_k \circ (\theta_k)^j}{k} =\frac{\E (f_k|\mathcal{I}_k)}k,$$
où $\mathcal{I}_k$ est la tribu des invariants par $\theta_k$.
Contrôlons maintenant les termes résiduels.
Comme la famille finie $((f_r+g_r))_{0 \le r \le k-1}$ est équi-intégrable et que $\theta_k$ laisse $\P$ invariant, la famille $(\miniop{}{\sup}{0 \le r \le k-1}(f_r+g_r)\circ \theta_k^n)_{n \ge 1}$ est équi-intégrable, ce qui assure la convergence
presque sûre et dans $L^1$: 
\begin{equation*}
\label{reste}
\lim_{n \to +\infty} \frac{1}{n}\sup_{0 \le r \le k-1} (f_r+g_r)\circ (\theta_k)^n=0.
\end{equation*}
On a $\displaystyle \sum_{n=1}^{+\infty} \E[(\frac{r^+_{nk,r}}n)^{\alpha}]\le\sum_{n=1}^{+\infty}\frac{C_r}{n^{\alpha}}<+\infty,$
ce qui entraîne, comme précédemment, que $r_{nk,r}^+/n$ tend presque sûrement vers $0$.
Finalement, 
\begin{eqnarray*}
 \forall r\in\{0,\dots,k-1\}\quad \miniop{}{\limsup}{n \to +\infty}\frac{f_{nk+r}}{nk+r} & \le & \frac {\E [f_k|\mathcal{I}_k]}k,
\end{eqnarray*}
et donc $\displaystyle \E[\miniop{}{\limsup}{n \to +\infty}\frac{f_{n}}{n}] \le \frac {\E [f_k]}k.$
On achève la preuve en faisant tendre $k$ vers~$+\infty$.
\finpreuve \end{proof} 

\textbf{Remarque} Dans le cas où il n'y a pas de défaut de sous-additivité, les hypothèses du théorème~\ref{thergdeux} sont évidemment vérifiées;  on obtient ainsi un théorème ergodique sous-addititif qui ressemble beaucoup à celui de Liggett~\cite{MR806224} sans toutefois lui être strictement comparable, au sens où aucun des deux n'entraîne l'autre.

En effet, prolongeant une remarque faite par Kingman dans son cours de Saint-Flour~\cite{MR0438477} (page 178), on peut noter que l'hypothèse de l'article original de Kingman, à savoir la stationnarité du processus doublement indexé $(X_{s,t})_{s\ge 0,t\ge 0}$, peut être affaiblie de deux manières:
\begin{itemize}
\item Soit en supposant que pour tout $k$, le processus $(X_{(r-1)k,rk})_{r\ge 1}$ est stationnaire -- c'est l'hypothèse qui sera utilisée par Liggett~\cite{MR806224}.
\item Soit en supposant que la loi de $X_{n,n+p}$ ne dépend que de $p$. Cette hypothèse, suggérée par Hammersley et Welsh, est celle que nous faisons ici, ou que fait Schürger dans~\cite{MR1127716}.
\end{itemize}
Il faut noter toutefois que dans la preuve de Liggett~\cite{MR806224}, l'hypothèse spéciale de stationnarité n'intervient que dans la partie réputée facile du contrôle de la limite supérieure.

A l'époque, Kingman semble penser que le premier jeu d'hypothèses l'a emporté sur le second, au vu des applications possibles. Plus de trente ans plus tard, les progrès des théorèmes ergodiques sous-additifs -- en particulier s'agissant du contrôle de la limite inférieure -- amènent à nuancer cette affirmation.

\vspace{0.5cm}
{Les auteurs tiennent à remercier le rapporteur anonyme qui a décelé une erreur dans une première version de cet article.}


\def\refname{Références}
\bibliographystyle{plain}


\end{document}